\documentclass[12pt]{amsart}

\usepackage{amssymb,amsmath,amscd,enumerate,verbatim,xcolor,mathtools,fullpage}
\usepackage{thmtools}
\usepackage[colorlinks=true,linkcolor=blue,citecolor=blue, pagebackref=false]{hyperref}

\usepackage{float}
\usepackage[all]{xy}
\usepackage[capitalise]{cleveref}
\usepackage{graphicx}
\usepackage{tikz-cd}
\usepackage{multirow,tabularx}
\usepackage{graphicx}

\usepackage{caption}
\usepackage{subcaption}
\usepackage[numbers,sort&compress]{natbib}

\usepackage{tikz}
\usetikzlibrary{calc}

\usepackage{mathptmx}
\usepackage{tipa}

\numberwithin{equation}{section}

\newtheorem{thm}{Theorem}[section]
\newtheorem{lem}[thm]{Lemma}

\newtheorem{cor}[thm]{Corollary}
\newtheorem{prop}[thm]{Proposition}
\newtheorem{prob}[thm]{Problem}

\newtheorem{claim}{Claim}[thm]

\crefname{thm}{theorem}{theorems}
\crefname{lem}{lemma}{lemmas}
\crefname{prop}{proposition}{propositions}
\crefname{ex}{example}{examples}
\crefname{defn}{definition}{definitions}

\theoremstyle{definition}
\newtheorem{assm}[thm]{Assumption}
\newtheorem{construction}[thm]{Construction}
\newtheorem{defn}[thm]{Definition}
\newtheorem{ex}[thm]{Example}

\newtheorem{rem}[thm]{Remark}

\newtheorem*{thm*}{Theorem}

\DeclareMathOperator{\Cl}{Cl}

\DeclareMathOperator{\me}{me}

\DeclareMathOperator{\gin}{gin}

\DeclareMathOperator{\IN}{IN}
\DeclareMathOperator{\Inc}{Inc}
\DeclareMathOperator{\ind}{ind}

\DeclareMathOperator{\ini}{in}

\DeclareMathOperator{\HS}{HS}

\DeclareMathOperator{\spi}{sp}

\newcommand{\Z}{{\mathbb Z}}
\newcommand{\N}{{\mathbb N}}
\newcommand{\Q}{{\mathbb Q}}

\newcommand{\Fc}{{\mathcal F}}

\def\Ep{\emph{\Large\textepsilon}}
\def\Gc{{\mathcal G}}
\def\Hc{{\mathcal H}}

\def\FF{{\mathfrak F}}

\newcommand{\kk}{\Bbbk}

\newcommand\defas{\coloneqq}

\makeatletter
\@namedef{subjclassname@2020}{%
	\textup{2020} Mathematics Subject Classification}
\makeatother

\title{Invariant chains of graphs}

\author[D.T. Hoang]{Do Trong Hoang}
\address{Faculty of Mathematics and Informatics, Hanoi University of Science and Technology, 1 Dai Co Viet, Bach Mai, Hanoi, Vietnam}
\email{hoang.dotrong@hust.edu.vn}

\author[M. Koley]{Mitra Koley}
\address{School of Mathematics, Indian Institute of Science Education and Research, Thiruvananthapuram, Vithura, 695551, India}
\email{mitra@iisertvm.ac.in} 

\author[D.V. Le]{Dinh Van Le}
\address{Department of Mathematics, FPT University, Hanoi, Vietnam}
\email{dinhlv2@fe.edu.vn}

\subjclass[2020]{Primary 05E18; Secondary 05C15, 05C38, 05C69, 05C70, 05E45}

\keywords{$\Inc$-invariant chain of graphs, equivariant Noetherianity, independence complex, clique complex, chromatic number, matching number}

\begin{document}
	
	\begin{abstract}
		We initiate a systematic study of $\Inc$-invariant chains of graphs, the combinatorial counterparts of $\Inc$-invariant chains of edge ideals arising in the theory of equivariant Noetherianity. Such a chain consists of graphs on growing vertex sets whose edge sets are compatible with the action of the monoid of strictly increasing maps on the positive integers. We show that several associated combinatorial invariants exhibit rigid asymptotic behavior. The independence number eventually stabilizes, and every fixed entry of the $f$-vector and the $h$-vector of the independence complex is eventually linear. For clique complexes, every fixed entry of the $f$-vector is eventually polynomial, whereas the entries of the $h$-vector are eventually quasi-polynomial. Moreover, the clique and chromatic numbers are eventually quasi-linear, and their difference is eventually at most one. We also prove that the matching number eventually attains the maximal value $\lfloor n/2 \rfloor$.
		Finally, admissible and minimal paths eventually have lengths at most $3$ and $5$, respectively, and their maximal lengths stabilize. These results reveal strong asymptotic regularity in graph families governed by increasing symmetry.
	\end{abstract}

	\maketitle

	\section{Introduction}
	
	For $n\ge1$, let $R_n=\kk[x_1,\dots,x_n]$ be the polynomial ring in $n$ variables over a field $\kk$. A rapidly developing area of research concerns chains of ideals $I_n\subseteq R_n$ that are invariant under the action of the infinite general linear group, the infinite symmetric group, or the monoid $\Inc$ of strictly increasing maps on the positive integers; see, e.g.,
	\cite{AH07,Co67,Dr14,DEF,HS12,JLR20,LR20,LR24,NR17,NR19,SS16,SS17}.
	Related finiteness questions up to symmetry also arise in discrete geometry, convex optimization, and machine learning; see
	\cite{KLR22,L24,LR21,LRV26,LRV27,LC,LD}.
	The basic finiteness mechanism in this setting is equivariant Noetherianity. Although the polynomial ring
	\[
	R=\bigcup_{n\ge1}R_n=\kk[x_i\mid i\in\N]
	\]
	is not Noetherian, a celebrated result of Cohen \cite{Co67}
	implies that it is Noetherian up to the action of $\Inc$; see also
	\cite{AH07,HS12}. Consequently, every $\Inc$-invariant chain of ideals
	stabilizes: beyond a finite stage, the entire chain is determined by a single member up to increasing relabelings.
	
	This stabilization naturally leads to questions about the asymptotic behavior of numerical invariants. Among the invariants studied in this context are codimension, projective dimension,
	Castelnuovo--Mumford regularity, and Hilbert series; see, e.g.,
	\cite{Ga24,JLR20,LNNR1,LNNR2,LN22,Na,NR17,Mu,MR20,Ra}. In particular, it is conjectured that the regularity and projective dimension of $I_n$ are eventually linear functions of $n$ for every $\Inc$-invariant chain. Despite substantial progress, these conjectures remain open in general.
	
	Chains of edge ideals form an important testing ground for this theory. For such chains, algebraic questions about the ideals can be translated into combinatorial questions about the corresponding graphs. The regularity conjecture for $\Inc$-invariant chains of edge ideals was settled in \cite{HNT2024}: the regularity is eventually constant, with eventual value either $2$ or $3$. The asymptotic depth, and hence the asymptotic projective dimension, was subsequently determined in \cite{HHLNN}. The proofs of these results reveal unexpectedly rigid combinatorial and topological properties of the associated graphs. For example, their induced matching numbers are eventually at most $2$, and the reduced homology of their independence complexes can be described explicitly.
	
	These developments motivate a direct study of the corresponding graph chains. Informally, an $\Inc$-invariant chain of graphs is a sequence
	$\Gc=(G_n)_{n\ge1}$, with $V(G_n)=[n]$, such that every increasing
	relabeling of $G_n$ into $[m]$ is a subgraph of $G_m$. The precise definition is given in \Cref{subsec.InvChain}. Our guiding problem is the following.
	
	\begin{prob}
		\label{prob:main}
		Describe the asymptotic behavior of combinatorial invariants along an $\Inc$-invariant chain of graphs.
	\end{prob}
	
	We address this problem for independence and clique complexes, chromatic and matching numbers, and certain paths arising in the theory of binomial edge ideals. For simplicity, all the results summarized below are stated for chains in the reduced form of \Cref{assm:main}. A parameter that plays a central role throughout is the \emph{sparsity index} $\spi(\Gc)$, which records the smallest distance between the endpoints of an edge at the stability stage.
	
	Our first main result concerns independence complexes. We establish eventual linearity for their $f$-vectors, $h$-vectors, and numbers of facets.
	
	\begin{thm}[\Cref{thm:f-vector-IN,conj:IN-facet}]
		Let $\Gc=(G_n)_{n\ge1}$ be an $\Inc$-invariant chain of graphs. Denote by $\IN(G_n)$ and $\alpha(G_n)$ the independence
		complex and the independence number of $G_n$, respectively, and let
		$\nu(\IN(G_n))$ denote the number of facets of $\IN(G_n)$. Then:
		\begin{enumerate}
			\item $\alpha(G_n)$ is eventually constant.
			\item For each $i\ge0$, $f_i(\IN(G_n))$ is eventually a linear
			function of $n$.
			\item For each $i\ge0$, $h_i(\IN(G_n))$ is eventually a linear
			function of $n$.
			\item There exists a constant $C\in\Z$ such that
			\[
			\nu(\IN(G_n))=n+C
			\]
			for all sufficiently large $n$.
		\end{enumerate}
	\end{thm}
	
	The behavior of clique complexes is different: their dimensions do not
	stabilize but instead grow quasi-linearly. Nevertheless, their $f$-vectors and $h$-vectors satisfy precise recurrence relations.
	
	\begin{thm}[\Cref{thm:clique}]
		Let $\Gc=(G_n)_{n\ge1}$ be an $\Inc$-invariant chain of graphs with
		$\spi(\Gc)=s$. Denote by $\Cl(G_n)$ and $\omega(G_n)$ the clique complex and clique number of $G_n$, respectively. Then there exists $n_0\in\N$ such that, for every $n\ge n_0$, the following statements hold:
		\begin{enumerate}
			\item We have
			\[
			\omega(G_n)
			=\omega(G_{n_0})+\left\lfloor\frac{n-n_0}{s}\right\rfloor.
			\]
			\item For every $i\ge0$,
			\[
			f_i(\Cl(G_n))
			=f_i(\Cl(G_{n-1}))+f_{i-1}(\Cl(G_{n-s})).
			\]
			\item For every $i\ge0$,
			\[
			h_i(\Cl(G_n))=
			\begin{cases}
				h_i(\Cl(G_{n-1}))+h_{i-1}(\Cl(G_{n-s}))
				&\text{if }s\nmid(n-n_0),\\
				h_i(\Cl(G_{n-1}))-h_{i-1}(\Cl(G_{n-1}))
				+h_{i-1}(\Cl(G_{n-s}))
				& \text{if }s\mid(n-n_0).
			\end{cases}
			\]
		\end{enumerate}
	\end{thm}
	
	As a consequence, we show that for each $i\ge0$, the number $f_i(\Cl(G_n))$ is eventually a polynomial in $n$ of degree $i+1$ with leading coefficient $1/(i+1)!$, whereas $h_i(\Cl(G_n))$ is eventually a quasi-polynomial of degree at most $i$; see \Cref{cor:f-h-clique}.
	
	We next turn to graph coloring. For arbitrary graphs, the difference between the chromatic and clique numbers can be arbitrarily large \cite{My55}. Along an $\Inc$-invariant chain, however, the two invariants have the same asymptotic growth and eventually differ by at most one.
	
	\begin{thm}[\Cref{thm:chromatic,prop:gamma-omega}]
		Let $\Gc=(G_n)_{n\ge1}$ be an $\Inc$-invariant chain of graphs with
		$\spi(\Gc)=s$. Then there exists $n_0$ such that
		\[
		\gamma(G_n)
		=\gamma(G_{n_0})+\left\lfloor\frac{n-n_0}{s}\right\rfloor
		\quad\text{for all }n\ge n_0.
		\]
		Furthermore,
		\[
		\omega(G_n)\le\gamma(G_n)\le\omega(G_n)+1
		\quad\text{for all }n\gg0.
		\]
	\end{thm}
	
	Every matching in a graph on $n$ vertices has at most$\lfloor n/2\rfloor$ edges. The next result shows that along an $\Inc$-invariant chain, the matching number eventually attains this upper bound.
	
	\begin{thm}[\Cref{thm:matching-number}]
		Let $\Gc=(G_n)_{n\ge1}$ be an $\Inc$-invariant chain of graphs. Then
		\[
		\mu(G_n)=\left\lfloor\frac{n}{2}\right\rfloor
		\quad\text{for all }n\gg0.
		\]
	\end{thm}
	
	Finally, we consider admissible and minimal paths. These paths can be used to describe Gr\"obner bases and generic initial ideals of binomial edge ideals \cite{CDG,HHHKR,O11}. Our result gives uniform bounds on their lengths and shows that the maximal lengths eventually stabilize.
	
	\begin{thm}[\Cref{prop:ad-path-length}]
		Let $\Gc=(G_n)_{n\ge1}$ be an $\Inc$-invariant chain of graphs. Then for $n\gg0$, the following statements hold:
		\begin{enumerate}
			\item 
			Every admissible path in $G_n$ has length at most $3$.
			\item 
			Every minimal path in $G_n$ has length at most $5$.
			\item 
			The maximal lengths of admissible and minimal paths in $G_n$ are constant.
		\end{enumerate}
	\end{thm}
	
	Taken together, these results reveal several distinct asymptotic regimes: stabilization for independence numbers and maximal path lengths, polynomial and quasi-polynomial growth for face data, quasi-linear growth for clique and chromatic numbers, and extremal growth for matching numbers. In each case, Inc-invariance turns an a priori unbounded sequence of graphs into a family governed by finitely many parameters. This work thus provides a combinatorial counterpart to the theory of $\Inc$-invariant ideals and further illustrates the interplay between infinite symmetry and stabilization.
	
	The paper is organized as follows. In \Cref{subsec.pre}, we collect the necessary background on graphs and simplicial complexes. In
	\Cref{subsec.InvChain}, we develop the basic properties of $\Inc$-invariant chains and introduce the numerical parameters used throughout the paper. Independence complexes, clique complexes, chromatic numbers, and matching numbers are studied in \Cref{subsec.indep,subsec.cliq,subsec.chro,subsec.mat},
	respectively. In \Cref{subsec.path}, we study admissible and minimal paths and their applications to binomial edge ideals. The Appendix contains proofs of several technical results used earlier in the paper.

	\section{Preliminaries}\label{subsec.pre}
	
	In this section, we collect the notation and terminology used throughout the paper. We refer to \cite{BM} for standard notions from graph theory and to \cite{BH98} for background on simplicial complexes and Stanley--Reisner rings.
	
	Let $\N$ denote the set of positive integers. For integers $m$ and $n$ with $m\ge n$, set
	\[
	[n,m]=\{n,n+1,\dots,m\}.
	\]
	In particular, we write $[m]=[1,m]=\{1,\dots,m\}$, and we use the convention that $[0]=\emptyset$. 
	
	%-----------------------------------------------------
	\subsection{Graphs and their numerical invariants}
	
	All graphs in this paper are simple. Let $G$ be such a graph. We denote its vertex set by $V(G)$ and its edge set by $E(G)\subseteq\binom{V(G)}{2}$. An edge with endpoints $u$ and $v$ is denoted by $\{u,v\}$. When $V(G)\subseteq\N$ and $u<v$, we often identify $\{u,v\}$ with the ordered pair $(u,v)$.
	
	For two graphs $G$ and $H$, we write $G\subseteq H$ if
	\[
	V(G)\subseteq V(H)
	\quad\text{and}\quad
	E(G)\subseteq E(H).
	\]
	Their \emph{intersection} and \emph{union} are the graphs defined by
	\begin{align*}
		V(G\cap H)&=V(G)\cap V(H),
		& E(G\cap H)&=E(G)\cap E(H),\\
		V(G\cup H)&=V(G)\cup V(H),
		& E(G\cup H)&=E(G)\cup E(H).
	\end{align*}
	For a set $U$, let $K_U$ denote the complete graph on $U$; when $U=[n]$, we simply write $K_n$. If $U\subseteq V(G)$, the subgraph of $G$ \emph{induced} by $U$ is
	\[
	G[U]=G\cap K_U.
	\]
	For $W\subseteq V(G)$, the graph obtained by deleting the vertices in $W$ is
	\[
	G\setminus W=G[V(G)\setminus W].
	\]

	The \emph{complement} of $G$, denoted by $G^c$, is the graph on $V(G)$ with edge set
	\[
	E(G^c)=\binom{V(G)}{2}\setminus E(G).
	\]
	Thus, for every $U\subseteq V(G)$, we have
	\[
	(G\setminus U)^c=G^c\setminus U.
	\]
	
	For a vertex $u\in V(G)$, its \emph{open neighborhood} and \emph{closed neighborhood} are, respectively,
	\[
	N_G(u)=\{v\in V(G)\mid\{u,v\}\in E(G)\}
	\quad\text{and}\quad
	N_G[u]=N_G(u)\cup\{u\}.
	\]

	A \emph{path} of length $q\ge 1$ in $G$ is a sequence of vertices
	\[
	P=(u_0,u_1,\dots,u_q)
	\]
	such that $\{u_{i-1},u_i\}\in E(G)$ for $1\le i\le q$. A \emph{cycle} of length $q\ge3$ is a path
	\[
	C_q=(u_1,u_2,\dots,u_q,u_1)
	\]
	such that $u_1,\dots,u_q$ are distinct. A path or cycle is \emph{induced} if the subgraph induced by its vertices contains no edges other than those belonging to the path or cycle. A graph $G$ is \emph{weakly chordal} if neither $G$ nor $G^c$ contains an induced cycle of length at least $5$.
	
	A subset $U\subseteq V(G)$ is a \emph{clique} if every two distinct vertices in $U$ are adjacent, and it is an \emph{independent set} if no two distinct vertices in $U$ are adjacent. The maximum cardinalities of a clique and an independent set are called the \emph{clique number} and the \emph{independence number} of $G$, denoted by $\omega(G)$ and $\alpha(G)$, respectively. By definition,
	$\alpha(G)=\omega(G^c).$
	
	For $k\in\N$, a \emph{$k$-coloring} of $G$ is a map $c\colon V(G)\to[k]$. It is \emph{proper} if $c(u)\ne c(v)$ for every $\{u,v\}\in E(G)$. The \emph{chromatic number} of $G$, denoted by $\gamma(G)$, is the least $k$ for which $G$ admits a proper $k$-coloring. Equivalently, $\gamma(G)$ is the minimum number of independent sets whose union is $V(G)$. Since the vertices of a clique must receive distinct colors, one always has
	\begin{equation}
		\label{eq:clique-chromatic-bound}
		\omega(G)\le\gamma(G).
	\end{equation}
	
	A graph $G$ is \emph{perfect} if
	%\[
	$\gamma(H)=\omega(H)$
	%\]
	for every induced subgraph $H$ of $G$. The Strong Perfect Graph Theorem, proved by Chudnovsky et al. \cite{CRST} (see also \cite[Theorem 14.18]{BM}), states that a graph is perfect if and only if neither the graph nor its complement contains an induced odd cycle of length at least $5$. Consequently, every weakly chordal graph is perfect.
	
	A \emph{matching} in $G$ is a set of pairwise disjoint edges. The maximum cardinality of a matching in $G$ is the \emph{matching number} of $G$, denoted by $\mu(G)$. If $|V(G)|=n$, then every matching covers twice as many vertices as it has edges, and hence
	\begin{equation}
		\label{eq:matching-upper-bound}
		\mu(G)\le\left\lfloor\frac{n}{2}\right\rfloor.
	\end{equation}
	Finally, if $H\subseteq G$, then
	\[
	\omega(H)\le\omega(G),
	\qquad
	\gamma(H)\le\gamma(G),
	\qquad
	\mu(H)\le\mu(G).
	\]
	
	%-----------------------------------------------------
	\subsection{Simplicial complexes}
	
	A \emph{simplicial complex} $\Delta$ on a finite vertex set $V(\Delta)$ is a collection of subsets of $V(\Delta)$ such that, whenever $F'\in\Delta$ and $F\subseteq F'$, one has $F\in\Delta$. The elements of $\Delta$ are called \emph{faces}, and its maximal faces under inclusion are called \emph{facets}. We denote the set of facets of $\Delta$ by $\FF(\Delta)$.
	The \emph{dimension} of a face $F\in\Delta$ is $\dim F=|F|-1$, and the dimension of $\Delta$ is
	\[
	\dim\Delta=\max\{\dim F\mid F\in\Delta\}.
	\]
	Suppose that $\dim\Delta=d-1$. The \emph{$f$-vector} of $\Delta$ is
	$(f_{-1},f_0,\dots,f_{d-1}),$
	where $f_i=f_i(\Delta)$ is the number of $i$-dimensional faces of $\Delta$; in particular, $f_{-1}=1$ counts the empty face. The \emph{Euler characteristic} and \emph{reduced Euler characteristic} of $\Delta$ are defined by
	\[
	\chi(\Delta)=\sum_{i=0}^{d-1}(-1)^if_i
	\quad\text{and}\quad
	\widetilde{\chi}(\Delta)=\sum_{i=-1}^{d-1}(-1)^if_i=\chi(\Delta)-1.
	\]
	
	Every graph gives rise to two simplicial complexes that will be central to this paper. The \emph{clique complex} of $G$, denoted by $\Cl(G)$, is the simplicial complex whose faces are the cliques of $G$. The \emph{independence complex} of $G$, denoted by $\IN(G)$, is the simplicial complex whose faces are the independent sets of $G$. Both complexes have vertex set $V(G)$, and
	\[
	\IN(G)=\Cl(G^c).
	\]
	Moreover,
	\begin{equation}
		\label{eq:dimensions-graph-complexes}
		\dim\Cl(G)=\omega(G)-1
		\quad\text{and}\quad
		\dim\IN(G)=\alpha(G)-1.
	\end{equation}
	
	%-----------------------------------------------------
	\subsection{Stanley--Reisner rings and edge ideals}
	
	Let $\Delta$ be a simplicial complex of dimension $d-1$ on vertex set $V(\Delta)$, and let $\kk$ be a field. Set
	\[
	R=\kk[x_v\mid v\in V(\Delta)].
	\]
	For $F\subseteq V(\Delta)$, write $x_F=\prod_{v\in F}x_v$. The \emph{Stanley--Reisner ideal} of $\Delta$ is the squarefree monomial ideal
	\[
	I_\Delta=\langle x_F \mid F \notin \Delta\rangle\subseteq R,
	\]
	and the \emph{Stanley--Reisner ring} of $\Delta$ is the quotient ring
	$\kk[\Delta]=R/I_\Delta.$
	Its Krull dimension is
	\[
	\dim\kk[\Delta]=\dim\Delta+1=d.
	\]
	The standard grading of $R$ induces a grading
	$\kk[\Delta]=\bigoplus_{q\ge0}\kk[\Delta]_q.$
	The \emph{Hilbert function} and \emph{Hilbert series} of $\kk[\Delta]$ are
	\[
	H(\kk[\Delta],q)=\dim_\kk\kk[\Delta]_q
	\quad\text{and}\quad
	\HS_{\kk[\Delta]}(t)
	=\sum_{q\ge0}H(\kk[\Delta],q)t^q,
	\]
	respectively. The Hilbert series is determined by the $f$-vector through the identity
	\begin{equation}
		\label{eq:f-vector}
		\HS_{\kk[\Delta]}(t)
		=\sum_{i=-1}^{d-1}
		\frac{f_i t^{i+1}}{(1-t)^{i+1}};
	\end{equation}
	see, e.g., \cite[Theorem~5.1.7]{BH98}.
	Writing the Hilbert series in the form
	\begin{equation}
		\label{eq:h-vector}
		\HS_{\kk[\Delta]}(t)
		=\frac{h_0+h_1t+\cdots+h_dt^d}{(1-t)^d},
	\end{equation}
	we call
	\[
	h(\Delta;t)=h_0+h_1t+\cdots+h_dt^d
	\]
	the \emph{$h$-polynomial} of $\Delta$ and $(h_0,h_1,\dots,h_d)$ its \emph{$h$-vector}. Comparing \eqref{eq:f-vector} and \eqref{eq:h-vector} gives
	\begin{equation}
		\label{eq:f-and-h-vector}
		h_j =\sum_{i=0}^j (-1)^{j-i} \binom{d-i}{j-i}f_{i-1} \quad\text{and}\quad
		f_{j-1} =\sum_{i=0}^j \binom{d-i}{j-i}h_{i} 
	\end{equation}
	for $j=0,\dots,d.$ In particular,
	\begin{equation}
		\label{eq:top-h-euler}
		h_d
		=\sum_{i=0}^d(-1)^{d-i}f_{i-1}
		=(-1)^{d-1}\widetilde{\chi}(\Delta).
	\end{equation}
	
	Let $G$ be a graph on $[n]$, and let $R=\kk[x_1,\dots,x_n]$. The \emph{edge ideal} of $G$ is the squarefree quadratic monomial ideal
	\[
	I(G)=\langle x_ix_j\mid\{i,j\}\in E(G)\rangle\subseteq R.
	\]
	Since the minimal nonfaces of $\IN(G)$ are precisely the edges of $G$, while those of $\Cl(G)$ are the edges of $G^c$, we have
	\[
	I_{\IN(G)}=I(G)
	\quad\text{and}\quad
	I_{\Cl(G)}=I(G^c).
	\]
	Consequently,
	\[
	R/I(G)=\kk[\IN(G)],
	\qquad
	\dim R/I(G)=\alpha(G),
	\]
	and similarly $R/I(G^c)=\kk[\Cl(G)]$ has dimension $\omega(G)$. We use the abbreviation
	\[
	\HS_G(t)\defas\HS_{R/I(G)}(t).
	\]
	
	We conclude with the standard deletion formula for this Hilbert series; see \cite[Theorem~5.1]{Wa92} and \cite[Proposition~1.20]{Va98}.
	
	\begin{lem}
		\label{lem:Hilbert-series-graph}
		Let $G$ be a graph and let $u\in V(G)$. Then
		\[
		\HS_G(t)
		=\HS_{G\setminus u}(t)
		+\frac{t}{1-t}\HS_{G\setminus N_G[u]}(t).
		\]
	\end{lem}

	%-----------------------------------------------------	 	
	\section{Invariant chains of graphs}
	\label{subsec.InvChain}
	
	In this section, we introduce invariant chains of graphs and establish their basic properties. 
	
	%-----------------------------------------------------
	\subsection{The monoid Inc}
	The \emph{monoid of strictly increasing maps on $\N$} is defined by
	\[
	\Inc=\{\pi\colon\N\longrightarrow\N\mid \pi(n)<\pi(n+1)\text{ for all }n\ge1\}.
	\]
	For positive integers $n\le m$, consider the subset
	\[
	\Inc_{n,m}=\{\pi\in\Inc\mid \pi(n)\le m\}.
	\]
	For every $k\in\Z_{\ge0}$, the map $\sigma_k\in\Inc$ given by
	\begin{equation}
		\label{eq.sigma}
		\sigma_k(i)=
		\begin{cases}
			i,   & \text{if }1\le i\le k,\\
			i+1, & \text{if }i\ge k+1
		\end{cases}
	\end{equation}
	belongs to $\Inc_{n,n+1}$ for every $n\ge1$. The next factorization slightly generalizes \cite[Lemma~3.2]{KLR22}; see also \cite[Proposition~4.6]{NR17}.
	
	\begin{lem}
		\label{lem:Inc-decomposition}
		For all positive integers $r\le n\le m$, one has
		\[
		\Inc_{r,m}=\Inc_{n,m}\circ\Inc_{r,n}.
		\]
	\end{lem}
	
	\begin{proof}
		The cases $n=r$ and $n=m$ follow immediately from the fact that the identity map belongs to $\Inc_{n,n}$. Assume that $r<n<m$. By \cite[Lemma~3.2]{KLR22},
		\[
		\Inc_{a,m}=\Inc_{a+1,m}\circ\Inc_{a,a+1}
		\quad\text{for }r\le a<m.
		\]
		Applying this identity successively for $a=r,r+1,\dots,n-1$ gives
		\[
		\Inc_{r,m}
		=\Inc_{n,m}\circ\Inc_{n-1,n}\circ\cdots\circ\Inc_{r,r+1}
		=\Inc_{n,m}\circ\Inc_{r,n},
		\]
		as desired.
	\end{proof}
	
	%-----------------------------------------------------
	\subsection{Invariant chains of graphs}
	
	Each $\pi\in\Inc$ induces a map $\pi\colon\N^2\longrightarrow\N^2$ defined by
	\[
	\pi((i,j))=(\pi(i),\pi(j))
	\quad\text{for }(i,j)\in\N^2.
	\]
	Let $G=(V,E)$ be a graph with $V\subseteq\N$. We identify an edge $\{i,j\}\in E$, where $i<j$, with the ordered pair $(i,j)\in\N^2$ and, by abuse of notation, write $(i,j)\in E$. We denote by $\pi(G)$ the graph with vertex set $\pi(V)$ and edge set $\pi(E)$. More generally, for $\Gamma\subseteq\Inc$, let $\Gamma(G)$ be the graph with vertex and edge sets
	\[
	\Gamma(V)=\bigcup_{\pi\in\Gamma}\pi(V)
	\quad\text{and}\quad
	\Gamma(E)=\bigcup_{\pi\in\Gamma}\pi(E),
	\]
	respectively.
	
	\begin{defn}
		Let $\mathcal{G}=(G_n)_{n\ge1}$ be a sequence of graphs such that $V(G_n)\subseteq[n]$ for all $n\ge1$. We say that $\mathcal{G}$ is an \emph{$\Inc$-invariant chain of graphs} if
		\[
		\Inc_{n,m}(G_n)\subseteq G_m
		\quad\text{for all }m\ge n\ge1.
		\]
		The chain $\mathcal{G}$ \emph{stabilizes} if there exists $r\in\N$ such that
		\[
		\Inc_{n,m}(G_n)=G_m
		\quad\text{for all }m\ge n\ge r.
		\]
		The smallest such integer $r$ is called the \emph{stability index} of $\mathcal{G}$ and is denoted by $\ind(\mathcal{G})$.
	\end{defn}
	
	\begin{rem}
		A chain of graphs $\mathcal{G}=(G_n)_{n\ge1}$ is $\Inc$-invariant if and only if the corresponding chain of edge ideals $\mathcal{I}=(I(G_n))_{n\ge1}$ is $\Inc$-invariant. For background on $\Inc$-invariant chains of ideals, we refer the reader to \cite{HS12,JLR20,LNNR1,LNNR2,LN22,NR17}; see \cite{HNT2024,HHLNN} for further results on $\Inc$-invariant chains of edge ideals. A classical result of Cohen \cite{Co67} on the equivariant Noetherianity of the polynomial ring $\kk[x_i\mid i\in\N]$ (see also \cite{AH07,HS12}) implies that every $\Inc$-invariant chain of ideals stabilizes. Consequently, every $\Inc$-invariant chain of graphs stabilizes as well.
	\end{rem}
	
	\begin{ex}
		\label{ex:complete-chain}
		Let $K_n$ denote the complete graph on $[n]$. Then $\mathcal{K}=(K_n)_{n\ge1}$ is an $\Inc$-invariant chain with $\ind(\mathcal{K})=2$.
	\end{ex}
	
	The following characterization of the stability index is an immediate consequence of \Cref{lem:Inc-decomposition}; a more general result appears in \cite[Lemma~5.2]{NR17}.
	
	\begin{lem}
		\label{lem:ind}
		Let $\mathcal{G}=(G_n)_{n\ge1}$ be an $\Inc$-invariant chain of graphs. The following conditions are equivalent:
		\begin{enumerate}
			\item $\ind(\mathcal{G})\le r$;
			\item $\Inc_{r,m}(G_r)=G_m$ for all $m\ge r$;
			\item $\Inc_{m,m+1}(G_m)=G_{m+1}$ for all $m\ge r$.
		\end{enumerate}
	\end{lem}
	
	This characterization yields a simple construction that will occasionally be used to produce a chain with a prescribed stability index.
	
	\begin{construction}
		\label{cons:stab-index}
		Let $\mathcal{G}=(G_n)_{n\ge1}$ be an $\Inc$-invariant chain of graphs with $\ind(\mathcal{G})\le r$, and assume that $G_r\ne\emptyset$. Define $\widetilde{\Gc}=(\widetilde{G}_n)_{n\ge1}$ by
		\[
		\widetilde{G}_n=
		\begin{cases}
			\emptyset & \text{if }1\le n<r,\\
			G_n       & \text{if }n\ge r.
		\end{cases}
		\]
		Then $\widetilde{\Gc}$ is an $\Inc$-invariant chain with $\ind(\widetilde{\Gc})=r$.
	\end{construction}
	
	To avoid unnecessary technicalities, we impose the following standing assumption throughout the remainder of the paper, except in \Cref{def:me,def:si} below.
	
	\begin{assm}
		\label{assm:main}
		Let $\mathcal{G}=(G_n)_{n\ge1}$ be an $\Inc$-invariant chain of graphs with $\ind(\mathcal{G})=r$. We assume that $V(G_r)=[r]$ and that $(1,i),(j,r)\in E(G_r)$ for some $i,j\in[r]$. In particular, $V(G_n)=[n]$ for all $n\ge r$.
	\end{assm}
	
	Any $\Inc$-invariant chain satisfying $E(G_n)\ne\emptyset$ for $n\gg0$ can be reduced to this setting by deleting isolated vertices and relabeling the remaining vertices, as the following example illustrates.
	
	\begin{ex}
		Let $\mathcal{G}=(G_n)_{n\ge1}$ be an $\Inc$-invariant chain of graphs with $\ind(\mathcal{G})=7$ such that
		\[
		V(G_7)=[7]\setminus\{3\}
		\quad\text{and}\quad
		E(G_7)=\{(2,5),(4,6)\}.
		\]
		Then $V(G_n)=[n]$, and $1$ and $n$ are isolated vertices of $G_n$ for all $n\ge8$.
		
		Consider the $\Inc$-invariant chain $\mathcal{G}'=(G_n')_{n\ge1}$ with $\ind(\mathcal{G}')=5$, where
		\[
		V(G_5')=[5]
		\quad\text{and}\quad
		E(G_5')=\{(1,4),(3,5)\}.
		\]
		Although $G_5'\not\cong G_7\setminus\{1,7\}$, since the latter has only four vertices, it is clear that $G_6'\cong G_{8}\setminus\{1,8\}$, and hence
		\[
		G_{n-2}'\cong G_n\setminus\{1,n\}
		\quad\text{for all }n\ge8.
		\]
		Thus, instead of studying $\Gc$, we may study the chain $\Gc'$, which satisfies \Cref{assm:main}.
	\end{ex}
	
	We next record several basic properties of $\Inc$-invariant chains of graphs that will be used throughout the paper. The following observation was proved in \cite[Lemma~3.3]{HNT2024}.
	
	\begin{lem}
		\label{lem:E(Gn)}
		Let $r\in\N$ and let $G_r$ be a graph on $[r]$. For $n\ge r$, an ordered pair $(k,l)$ belongs to $\Inc_{r,n}(E(G_r))$ if and only if there exists $(i,j)\in E(G_r)$ such that
		\begin{equation}
			\label{eq:inequalities}
			0\le k-i\le l-j\le n-r.
		\end{equation}
		In particular, the following statements hold:
		\begin{enumerate}
			\item $\Inc_{r,n}(G_r)=\Inc_{r+a,n+a}(G_r)$ for all $a\in\Z_{\ge0}$;
			\item if $\mathcal{G}=(G_n)_{n\ge1}$ is an $\Inc$-invariant chain of graphs with $\ind(\mathcal{G})=r$, then $(k,l)\in E(G_n)$ for $n\ge r$ if and only if \eqref{eq:inequalities} holds for some $(i,j)\in E(G_r)$.
		\end{enumerate}
	\end{lem}
	
	Geometrically, the inequalities in \eqref{eq:inequalities} mean that $(k,l)\in\mathbb{R}^2$ lies in the isosceles right triangle with vertices $(i,j)$, $(i,j+n-r)$, and $(i+n-r,j+n-r)$. Using this interpretation, it was shown in \cite[Lemma~3.7]{HHLNN} that certain edges of $G_n$, viewed as points in $\mathbb{R}^2$, can be moved by small integral steps in any of the four cardinal directions while remaining edges of $G_n$ (see \Cref{fig:moves}):
	
	\begin{lem}
		\label{lem.shortmoves}
		Let $\mathcal{G}=(G_n)_{n\ge1}$ be an $\Inc$-invariant chain of graphs with $\ind(\mathcal{G})=r$, and let $n\ge r$. Then the following statements hold:
		\begin{enumerate}
			\item \textup{(East and south moves)} Assume that $n\ge3r$. Let $k\le k'\le r$ and $n-r\le l'\le l$ be integers. If $(k,l)\in E(G_n)$, then $(k,l'),(k',l)\in E(G_n)$.
			\item \textup{(West move)} Let $r\le k''\le k<l$ be integers. If $(k,l)\in E(G_n)$, then $(k'',l)\in E(G_n)$.
			\item \textup{(North move)} Let $k<l\le l''\le n-r$ be integers. If $(k,l)\in E(G_n)$, then $(k,l'')\in E(G_n)$.
		\end{enumerate}
	\end{lem}
	
	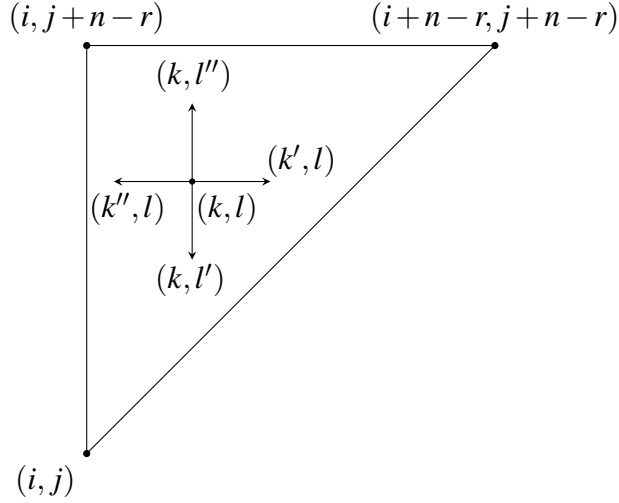
\begin{figure}[H]
		\centering
		\begin{tikzpicture}[>=stealth,scale=0.9]
			% Triangle
			\coordinate (A) at (0,0);
			\coordinate (B) at (0,6);
			\coordinate (C) at (6,6);
			
			\draw (A)--(B)--(C);
			\draw (A)--(C);
			
			\fill (A) circle (1.5pt);
			\fill (B) circle (1.5pt);
			\fill (C) circle (1.5pt);
			
			\node[below left] at (A) {$(i,j)$};
			\node[above] at (B) {$(i,j+n-r)$};
			\node[above] at (C) {$(i+n-r,j+n-r)$};
			
			% Center
			\coordinate (P) at (1.55,4.0);
			\fill (P) circle (1.3pt);
			
			% Arrows
			\draw[-stealth] (P)--++(1.15,0);
			\draw[-stealth] (P)--++(-1.15,0);
			\draw[-stealth] (P)--++(0,1.15);
			\draw[-stealth] (P)--++(0,-1.15);
			
			% Labels
			\node[right=-0.1cm] at ($(P)+(0,-0.4)$) {$(k,l)$};
			\node[right=-1.6cm] at ($(P)+(0.1,-0.4)$) {$(k'',l)$};
			\node[right=1.1cm] at ($(P)+(-0.3,0.3)$) {$(k',l)$};
			\node[right=-0.95cm] at ($(P)+(0.35,1.55)$) {$(k,l'')$};
			\node[right=-0.95cm] at ($(P)+(0.35,-1.45)$) {$(k,l')$};
		\end{tikzpicture}
		\caption{Moving an edge in the four cardinal directions}
		\label{fig:moves}
	\end{figure}

	We will frequently use the following refinement of \Cref{lem.shortmoves}(i).
	
	\begin{lem}
		\label{lem:gap}
		Let $\mathcal{G}=(G_n)_{n\ge1}$ be an $\Inc$-invariant chain of graphs with $\ind(\mathcal{G})=r$. Suppose that $n\ge r$, and let $u,v\in V(G_n)$ satisfy $v\ge u+r$. If there exists $(k,l)\in E(G_n)$ such that $k\le u<v\le l$, then $(u,v)\in E(G_n)$.
	\end{lem}
	
	\begin{proof}
		Since $(k,l)\in E(G_n)$, \Cref{lem:E(Gn)} yields an edge $(i,j)\in E(G_r)$ such that
		\[
		0\le k-i\le l-j\le n-r.
		\]
		As $k\le u<v\le l$, we obtain $0\le u-i$ and $v-j\le n-r$. Moreover,
		\[
		v-u\ge r\ge j-i,
		\]
		so $u-i\le v-j$. Therefore,
		\[
		0\le u-i\le v-j\le n-r,
		\]
		and \Cref{lem:E(Gn)} implies that $(u,v)\in E(G_n)$.
	\end{proof}

	%-----------------------------------------------------
	
	\subsection{Number of missing edges and sparsity index}
	
	We introduce two numerical invariants of $\Inc$-invariant chains of graphs that will play an important role throughout the paper. In \Cref{def:me,def:si}, we do \emph{not} impose \Cref{assm:main}, because certain inductive arguments, such as the proof of \Cref{lem:induced-chain}, require the slightly more general case $V(G_r)=[a,r]$ for some $a\in[r]$.
	
	The first invariant measures the “distance" from a chain to the chain of complete graphs by counting the edges missing at the stability index.

	\begin{defn}
		\label{def:me}
		Let $\Gc=(G_n)_{n\ge1}$ be an $\Inc$-invariant chain of graphs with $\ind(\Gc)=r$. The \emph{number of missing edges} of $\Gc$ is
		\[
		\me(\Gc)
		=|E(G_r^c)|
		=|E(K_{V(G_r)})|-|E(G_r)|
		=\binom{|V(G_r)|}{2}-|E(G_r)|.
		\]
	\end{defn}
	
	The second invariant, introduced in \cite{HHLNN}, measures the combinatorial complexity of the chain.
	
	\begin{defn}
		\label{def:si}
		Let $\Gc=(G_n)_{n\ge1}$ be an $\Inc$-invariant chain of graphs with $\ind(\Gc)=r$, and assume that $E(G_r)\ne\emptyset$. The \emph{sparsity index} of $\Gc$ is
		\[
		\spi(\Gc)=\min\{j-i\mid(i,j)\in E(G_r)\}.
		\]
	\end{defn}

	\begin{ex}
		Let $\mathcal{G}=(G_n)_{n\ge1}$ be an $\Inc$-invariant chain of graphs with $\ind(\mathcal{G})=5$ such that
		\[
		V(G_5)=[2,5]
		\quad\text{and}\quad
		E(G_5)=\{(2,4),(2,5)\}.
		\]
		Then 
		\[
		\me(\Gc)=\binom{4}{2}-|E(G_5)|=4
		\quad\text{and}\quad
		\spi(\Gc)=4-2=2.
		\]
	\end{ex}
	
	The sparsity index gives a simple criterion for determining whether certain pairs of vertices form edges.
	
	\begin{lem}
		\label{lem:independent-interval}
		Let $\Gc=(G_n)_{n\ge1}$ be an $\Inc$-invariant chain of graphs with $\ind(\Gc)=r$ and $\spi(\Gc)=s$. Then, for every $n\ge r$, the following statements hold:
		\begin{enumerate}
			\item if $k,l\in[r,n-r]$ and $l-k\ge s$, then $(k,l)\in E(G_n)$;
			\item if $k,l\in[n]$ and $0<l-k<s$, then $(k,l)\notin E(G_n)$. In particular, for every $t\in[n-s+1]$, the interval $I_t=[t,t+s-1]$ is an independent set of $G_n$.
		\end{enumerate}
	\end{lem}
	
	\begin{proof}
		(i) Since $\spi(\Gc)=s$, there exists $i\in[r]$ such that $(i,i+s)\in E(G_r)$. As $k,l\in[r,n-r]$ and $l-k\ge s$, we have
		\[
		0\le k-i\le l-(i+s)\le n-r.
		\]
		Hence, $(k,l)\in E(G_n)$ by \Cref{lem:E(Gn)}.
		
		(ii) The first assertion follows immediately from \eqref{eq:inequalities}. The second follows because the distance between any two vertices of $I_t$ is at most $s-1$.
	\end{proof}
	
	%-----------------------------------------------------	
	
	\subsection{Saturated chains}	
	
	Since $\Inc_{n,n+1}$ contains the identity map for every $n\ge1$, any $\Inc$-invariant chain of graphs $\mathcal{G}=(G_n)_{n\ge1}$ is increasing:
	\[
	G_1\subseteq G_2\subseteq\cdots\subseteq G_n\subseteq\cdots.
	\]
	The \emph{limit} of $\Gc$ is the graph $G_\infty$ with vertex and edge sets
	\[
	V(G_\infty)=\bigcup_{n\ge1}V(G_n)
	\quad\text{and}\quad
	E(G_\infty)=\bigcup_{n\ge1}E(G_n),
	\]
	respectively. Thus, $G_\infty$ is an infinite simple graph containing each $G_n$ as a subgraph. Moreover, the $\Inc$-invariance of $\Gc$ implies that $G_\infty$ is itself $\Inc$-invariant, that is, $\Inc(G_\infty)\subseteq G_\infty$. Since each $\Inc_{n,m}$ contains the identity map, it follows that
	\begin{equation}
		\label{eq:Inc-G-infty}
		\Inc_{n,m}(G_\infty)=\Inc(G_\infty)=G_\infty
		\quad\text{for all }m\ge n\ge1.
	\end{equation}
	The following characterization of $G_\infty$ is an immediate consequence of \Cref{lem:E(Gn)}.
	
	\begin{cor}
		\label{cor:G_infty}
		Let $\mathcal{G}=(G_n)_{n\ge1}$ be an $\Inc$-invariant chain of graphs with $\ind(\mathcal{G})=r$, and let $G_\infty$ be its limit. Then $(k,l)\in E(G_\infty)$ if and only if there exists $(i,j)\in E(G_r)$ such that
		\[
		0\le k-i\le l-j.
		\]
		In particular,
		\begin{equation*}
			\spi(\Gc)=\min\{j-i\mid(i,j)\in E(G_\infty)\}.
		\end{equation*}
	\end{cor}
	
	Among all $\Inc$-invariant chains with a fixed limit, there is a largest one.
	
	\begin{defn}
		Let $\mathcal{G}=(G_n)_{n\ge1}$ be an $\Inc$-invariant chain of graphs with limit $G_\infty$. The \emph{saturation} of $\Gc$ is the chain $\widehat{\Gc}=(\widehat{G}_n)_{n\ge1}$ defined by
		\[
		\widehat{G}_n\defas G_\infty[V(G_n)]=G_\infty\cap K_{V(G_n)}.
		\]
		We say that $\Gc$ is \emph{saturated} if $\Gc=\widehat{\Gc}$.
		More generally, $\Gc$ is \emph{eventually saturated} if
		\[
		G_n=\widehat{G}_n
		\quad\text{for all }n\ge\ind(\Gc).
		\]
	\end{defn}

	Evidently, $\widehat{\Gc}$ is an $\Inc$-invariant chain with $\widehat{G}_n\supseteq G_n$ for all $n\ge1$. Moreover, $\Gc$ and $\widehat{\Gc}$ have the same limit and hence the same
	sparsity index. In practice, saturated and eventually saturated chains are often easier to analyze and will serve as base cases in several inductive arguments. The following lemma characterizes these chains.
	
	\begin{lem}
		\label{lem:saturation-characterization}
		Let $\mathcal{G}=(G_n)_{n\ge1}$ be an $\Inc$-invariant chain of graphs with limit $G_\infty$. The following conditions are equivalent:
		\begin{enumerate}
			\item $\mathcal{G}$ is saturated \textup{(respectively, eventually saturated)};
			\item $G_{n+1}\cap K_{V(G_n)}=G_n$ for all $n\ge1$ \textup{(respectively, for all $n\ge\ind(\mathcal{G})$)}.
		\end{enumerate}
	\end{lem}
	
	\begin{proof}
		We prove the assertion for saturated chains; the eventually saturated case is identical after restricting to $n\ge\ind(\mathcal{G})$.
		
		(i)$\Rightarrow$(ii): If $\mathcal{G}$ is saturated, then
		\[
		G_n\subseteq G_{n+1}\cap K_{V(G_n)}
		\subseteq G_\infty\cap K_{V(G_n)}=G_n
		\]
		for every $n\ge1$. Hence, equality holds throughout.
		
		(ii)$\Rightarrow$(i): Assume that $G_{n+1}\cap K_{V(G_n)}=G_n$ for every $n\ge1$. Then
		\[
		G_{n+2}\cap K_{V(G_n)}
		=\bigl(G_{n+2}\cap K_{V(G_{n+1})}\bigr)\cap K_{V(G_n)}
		=G_{n+1}\cap K_{V(G_n)}
		=G_n.
		\]
		By induction, this yields
		\[
		G_m\cap K_{V(G_n)}=G_n
		\quad\text{for all }m\ge n\ge1.
		\]
		Consequently,
		\begin{align*}
			\widehat{G}_n
			&=G_\infty\cap K_{V(G_n)}
			=\biggl(\bigcup_{m\ge n}G_m\biggr)\cap K_{V(G_n)}
			=\bigcup_{m\ge n}\bigl(G_m\cap K_{V(G_n)}\bigr)
			=G_n
		\end{align*}
		for every $n\ge1$. Thus, $\mathcal{G}$ is saturated.
	\end{proof}

	In the remainder of this section, we state two technical results whose proofs are postponed to the Appendix. The first provides a way to decrease the missing-edge number $\me(\Gc)$, which will be used in several induction arguments.

	\begin{prop}
		\label{lem:induced-chain}
		Let $\Gc=(G_n)_{n\ge1}$ be an $\Inc$-invariant chain of nonempty graphs with $\ind(\Gc)=r$. Let $\alpha,\beta$ be integers with $1\le\alpha\le r$ and $0\le\beta\le r$. For $n\ge\alpha+\beta$, set $U_n=[\alpha,n-\beta]$, and define $\Hc=(H_n)_{n\ge1}$ by
		\[
		H_n=
		\begin{cases}
			\emptyset, & \text{if }1\le n\le r+\alpha-2,\\
			G_{n+\beta}\cap K_{U_{n+\beta}}, & \text{if }n\ge r+\alpha-1.
		\end{cases}
		\]
		Then $\Hc$ is an $\Inc$-invariant chain satisfying the following properties:
		\begin{enumerate}
			\item $\ind(\Hc)=r+\alpha-1$;
			\item $\spi(\Hc)=\spi(\Gc)$;
			\item $\me(\Hc)\le\me(\Gc)$. Moreover, if $\beta\ge1$ and equality holds, then $\Gc$ is eventually saturated. 
		\end{enumerate}
	\end{prop}
	
	The second result identifies the graphs obtained by deleting a sufficiently long interval from two consecutive members of the chain. It generalizes \cite[Lemma~3.10(ii)]{HHLNN}. 
	
	\begin{prop}
		\label{lem:del-graph-iso}
		Let $\Gc=(G_n)_{n\ge1}$ be an $\Inc$-invariant chain of graphs with $\ind(\Gc)=r$. Let $n\ge r+1$, and let $\alpha,\beta\in[n]$ satisfy $\beta-\alpha\ge r$. Then, for every $k\in[\alpha,\beta]$, the map $\sigma_k$ induces an isomorphism
		\[
		G_n\setminus[\alpha,\beta]
		\xrightarrow{\ \cong\ }
		G_{n+1}\setminus[\alpha,\beta+1].
		\]
	\end{prop}

	\section{Independence complexes}
	\label{subsec.indep}
	
	In this section, we investigate the asymptotic behavior of independence complexes along an $\Inc$-invariant chain of graphs. Our main objective is to prove the following generalization of \cite[Proposition~7.5]{HHLNN}, conjectured by Hop D. Nguyen in a private communication.
	
	\begin{thm}
		\label{thm:f-vector-IN}
		Let $\Gc=(G_n)_{n\ge1}$ be an $\Inc$-invariant chain of graphs with $\ind(\Gc)=r$, $\spi(\Gc)=s$, and $\me(\Gc)=m$. Denote by $\IN(G_n)$ and $\alpha(G_n)$ the independence complex and the independence number of $G_n$, respectively. Let
		\[
		\bigl(f_{-1}(\IN(G_n)),f_0(\IN(G_n)),f_1(\IN(G_n)),\dots\bigr)
		\quad\text{and}\quad
		\bigl(h_0(\IN(G_n)),h_1(\IN(G_n)),\dots\bigr)
		\]
		be the $f$-vector and the $h$-vector of $\IN(G_n)$, respectively. Then, for every $n\ge2r+m$, the following statements hold:
		\begin{enumerate}
			\item $\alpha(G_{n+1})=\alpha(G_n)$.
			\item For every $i\ge0$,
			\[
			f_i(\IN(G_{n+1}))
			=f_i(\IN(G_n))+\binom{s-1}{i}.
			\]
			In particular, $f_i(\IN(G_n))$ is eventually a linear function of $n$ for every $i\ge0$.
			\item Set $d=\alpha(G_n)$. Then for every $i\ge0$,
			\[
			h_i(\IN(G_{n+1}))
			=h_i(\IN(G_n))+(-1)^{i-1}\binom{d-s}{i-1}.
			\]
			In particular, $h_i(\IN(G_n))$ is eventually a linear function of $n$ for every $i\ge0$.
			\item Let $\Ep(\IN(G_n))$ denote either the ordinary or the reduced Euler characteristic of $\IN(G_n)$. Then
			\[
			\Ep(\IN(G_{n+1}))=
			\begin{cases}
				\Ep(\IN(G_n))+1 & \text{if }s=1,\\
				\Ep(\IN(G_n))   & \text{if }s\ge2.
			\end{cases}
			\]
		\end{enumerate}
	\end{thm}
	
	\begin{rem}
		The eventual constancy of $\alpha(G_n)$ also follows from \cite[Corollary~3.12]{LNNR1}. Indeed, the Stanley--Reisner ideal of $\IN(G_n)$ is the edge ideal $I(G_n)$, and the chain $(I(G_n))_{n\ge1}$ is $\Inc$-invariant. By \cite[Corollary~3.12]{LNNR1}, the dimension of the Stanley--Reisner ring of $\IN(G_n)$, which is precisely $\alpha(G_n)$, is eventually constant.
	\end{rem}
	
	The following example illustrates \Cref{thm:f-vector-IN}.
	
	\begin{ex}
		\label{ex:independence}
		Let $\mathcal{G}=(G_n)_{n\ge1}$ be an $\Inc$-invariant chain of graphs with $\ind(\mathcal{G})=6$ and
		\[
		E(G_6)=\{(1,6),(2,5),(3,6)\}.
		\]
		Then $\spi(\Gc)=3$. Computations with Macaulay2 \cite{GS} yield the following table.
		
		\begin{table}[htbp]
			\centering
			\begin{tabular}{|c|c|c|c|c|}
				\hline
				$n$ & $\alpha(G_n)$ & $f$-vector of $\IN(G_n)$ & $h$-vector of $\IN(G_n)$ & $\nu(\IN(G_n))$ \\
				\hline
				6  & 4 & $(1,6,12,9,2)$   & $(1,2,0,-1,0)$ & 4 \\
				\hline
				7  & 4 & $(1,7,13,9,2)$   & $(1,3,-2,0,0)$ & 4 \\
				\hline
				8  & 4 & $(1,8,15,10,2)$  & $(1,4,-3,0,0)$ & 5 \\
				\hline
				9  & 4 & $(1,9,17,11,2)$  & $(1,5,-4,0,0)$ & 6 \\
				\hline
				10 & 4 & $(1,10,19,12,2)$ & $(1,6,-5,0,0)$ & 7 \\
				\hline
			\end{tabular}
			\caption{Some invariants of independence complexes}
			\label{tab:independence}
		\end{table}
		
		The reader may verify that the conclusions of \Cref{thm:f-vector-IN} hold throughout the range $7\le n\le10$. The last column of the table records the number of facets of $\IN(G_n)$, which will be studied at the end of this section.
	\end{ex}
	
	Our proof of \Cref{thm:f-vector-IN} relies on the explicit relations connecting the $f$-vector, the $h$-vector, and the Hilbert series given in \eqref{eq:f-vector} and \eqref{eq:h-vector}. Under this algebraic framework, \Cref{thm:f-vector-IN} is equivalent to the following proposition.
	
	\begin{prop}
		\label{prop:HS-IN}
		Let $\Gc=(G_n)_{n\ge1}$ be an $\Inc$-invariant chain of graphs with $\ind(\Gc)=r$, $\spi(\Gc)=s$, and $\me(\Gc)=m$. Then
		\[
		\HS_{G_{n+1}}(t)=\HS_{G_n}(t)+\frac{t}{(1-t)^{s}} \quad \text{for all $n\ge 2r+m$}.
		\]
	\end{prop}
	
	To establish this proposition, we require an auxiliary result from \cite[Lemma~3.10]{HHLNN}, whose second assertion is a special case of \Cref{lem:del-graph-iso}. Henceforth, for an $\Inc$-invariant chain of graphs $\Gc=(G_n)_{n\ge1}$ with $\ind(\Gc)=r$, set
	\[
	\begin{aligned}
		b(\Gc)&=\min\{i\in[r]\mid(i,r)\in E(G_r)\},\\
		B(\Gc)&=\max\{i\in[r]\mid(i,r)\in E(G_r)\}.
	\end{aligned}
	\]
	These indices are well-defined by \Cref{assm:main}.
	
	\begin{lem}
		\label{lem.GWeaklyChordal}
		Let $\Gc=(G_n)_{n\ge1}$ be an $\Inc$-invariant chain of graphs with $\ind(\Gc)=r$. Then for every $n\ge 2r+1$, the following statements hold:
		\begin{enumerate}
			\item 
			$V(G_n\setminus N_{G_n}[n])=[b(\Gc)-1] \cup [n-r+B(\Gc)+1,n-1]$.
			\item 
			The graphs $G_n\setminus N_{G_n}[n]$ and $G_{n+1}\setminus N_{G_{n+1}}[n+1]$ are isomorphic.
			%\item $G_n\setminus N_{G_n}[n]$ is weakly chordal.
		\end{enumerate}
	\end{lem}
	
	As a first step toward proving \Cref{prop:HS-IN}, we treat eventually saturated chains.
	
	\begin{lem}
		\label{saturated}
		Let $\Gc=(G_n)_{n\ge1}$ be an eventually saturated $\Inc$-invariant chain of graphs with $\ind(\Gc)=r$ and $\spi(\Gc)=s$. Then the following statements hold:
		\begin{enumerate}
			\item 
			$V(G_n\setminus N_{G_n}[n])=[n-s+1,n-1]$ for all $n\ge 2r+1$.
			\item 
			We have
			\[
			\HS_{G_{n+1}}(t)=\HS_{G_n}(t)+\frac{t}{(1-t)^{s}} \quad \text{for all $n\ge 2r$}.
			\]
		\end{enumerate}	
	\end{lem}
	
	\begin{proof}
		(i) By \Cref{lem.GWeaklyChordal}(i), it suffices to show that $b(\Gc)=1$ and $B(\Gc)=r-s$. By \Cref{assm:main}, there exists $i\in[r]$ such that $(1,i)\in E(G_r)$. It follows from \Cref{cor:G_infty} that $(1,r)\in E(G_\infty)$. Since $\Gc$ is eventually saturated, $(1,r)\in E(G_r)$, and hence $b(\Gc)=1$.
		
		Since $\spi(\Gc)=s$, there exists an edge $(j-s,j)\in E(G_r)$. Again by \Cref{cor:G_infty}, we have $(r-s,r)\in E(G_\infty)$. Eventual saturation yields $(r-s,r)\in E(G_r)$, so $B(\Gc)\ge r-s$. On the other hand, $(B(\Gc),r)\in E(G_r)$ and $\spi(\Gc)=s$, whence $B(\Gc)\le r-s$. Therefore, $B(\Gc)=r-s$.
		
		(ii) Let $n\ge2r$. By part (i),
		\[
		V\bigl(G_{n+1}\setminus N_{G_{n+1}}[n+1]\bigr)=[n-s+2,n].
		\]
		The interval $[n-s+2,n]$ is independent in $G_{n+1}$ by \Cref{lem:independent-interval}(ii). Hence, $G_{n+1}\setminus N_{G_{n+1}}[n+1]$ is an edgeless graph on $s-1$ vertices, and therefore
		\[
		\HS_{G_{n+1}\setminus N_{G_{n+1}}[n+1]}(t)
		=\frac{1}{(1-t)^{s-1}}.
		\]
		Since $\Gc$ is eventually saturated, \Cref{lem:saturation-characterization} gives
		\[
		G_{n+1}\setminus\{n+1\}=G_{n+1}\cap K_n=G_n.
		\]
		It now follows from \Cref{lem:Hilbert-series-graph} that
		\begin{align*}
			\HS_{G_{n+1}}(t)
			&=\HS_{G_{n+1}\setminus\{n+1\}}(t)
			+\frac{t}{1-t}\HS_{G_{n+1}\setminus N_{G_{n+1}}[n+1]}(t)\\
			&=\HS_{G_n}(t)+\frac{t}{(1-t)^s},
		\end{align*}
		as desired.
	\end{proof}

	We are now ready to prove \Cref{prop:HS-IN}.
	
	\begin{proof}[Proof of \Cref{prop:HS-IN}]
		We proceed by induction on $m=\me(\Gc)$. If $m=0$, then $G_r=K_r$, and hence $G_n=K_n$ for all $n\ge r$. In particular, $\Gc$ is eventually saturated, so the assertion follows from \Cref{saturated}.
		
		Assume that $m>0$ and that the assertion holds for every chain $\Hc$ with $\me(\Hc)<m$. By \Cref{saturated}, we may further assume that $\Gc$ is not eventually saturated. Let $\Hc=(H_n)_{n\ge1}$ be the chain constructed in \Cref{lem:induced-chain} with $\alpha=\beta=1$; thus,
		\[
		H_n=
		\begin{cases}
			\emptyset       & \text{if }1\le n\le r-1,\\
			G_{n+1}\cap K_n & \text{if }n\ge r.
		\end{cases}
		\]
		By \Cref{lem:induced-chain}, we have  $\ind(\Hc)=r$, $\spi(\Hc)=s$, and $\me(\Hc)<m$.
		Let $n\ge 2r+1$. Since $H_n=G_{n+1}\setminus\{n+1\}$, \Cref{lem:Hilbert-series-graph} gives
		\[
		\HS_{G_{n+1}}(t)-\HS_{H_n}(t)
		=\frac{t}{1-t}\HS_{G_{n+1}\setminus N_{G_{n+1}}[n+1]}(t).
		\]
		By \Cref{lem.GWeaklyChordal}(ii), the graphs $G_n\setminus N_{G_n}[n]$ and $G_{n+1}\setminus N_{G_{n+1}}[n+1]$ are isomorphic. Therefore,
		\[
		\HS_{G_{n+1}}(t)-\HS_{H_n}(t)
		=\HS_{G_n}(t)-\HS_{H_{n-1}}(t)
		\quad\text{for all }n\ge2r+1,
		\]
		or equivalently,
		\[
		\HS_{G_{n+1}}(t)-\HS_{G_n}(t)
		=\HS_{H_n}(t)-\HS_{H_{n-1}}(t)
		\quad\text{for all }n\ge2r+1.
		\]
		Applying the induction hypothesis to $\Hc$, we obtain
		\[
		\HS_{H_n}(t)-\HS_{H_{n-1}}(t)
		=\frac{t}{(1-t)^s}
		\quad\text{for all }n\ge2r+\me(\Hc)+1.
		\]
		Since $\me(\Hc)<m$, the last two identities yield
		\[
		\HS_{G_{n+1}}(t)-\HS_{G_n}(t)
		=\frac{t}{(1-t)^s}
		\quad\text{for all }n\ge2r+m.
		\]
		This proves the proposition.
	\end{proof}
	
	Next, we deduce \Cref{thm:f-vector-IN} from \Cref{prop:HS-IN}.
	
	\begin{proof}[Proof of \Cref{thm:f-vector-IN}]
		Fix $n\ge 2r+m$.
		Set $f_i(\IN(G_n))=0$ for all $i\ge\alpha(G_n)$. Then \eqref{eq:f-vector} becomes
		\[
		\HS_{G_n}(t)
		=\sum_{i\ge-1}
		\frac{f_i(\IN(G_n))t^{i+1}}{(1-t)^{i+1}}.
		\]
		By \Cref{prop:HS-IN},
		\begin{equation}
			\label{eq:HS-difference}
			\HS_{G_{n+1}}(t)-\HS_{G_n}(t)
			=\frac{t}{(1-t)^s}.
		\end{equation}
		Consequently,
		\[
		\sum_{i\ge0}
		\bigl(f_i(\IN(G_{n+1}))-f_i(\IN(G_n))\bigr)
		\frac{t^{i+1}}{(1-t)^{i+1}}
		=\frac{t}{(1-t)^s}.
		\]
		Using the identity
		\[
		\frac{t}{(1-t)^s}
		=\sum_{i=0}^{s-1}\binom{s-1}{i}
		\frac{t^{i+1}}{(1-t)^{i+1}}
		\]
		and comparing the coefficients of the powers of $t/(1-t)$, we obtain
		\[
		f_i(\IN(G_{n+1}))
		=f_i(\IN(G_n))+\binom{s-1}{i}
		\quad\text{for every }i\ge0.
		\]
		This proves (ii).
		
		To establish (i), note that $\alpha(G_n)\ge s$. This holds because $[s]$ is an independent set of $G_n$ by \Cref{lem:independent-interval}(ii). Moreover, it follows from (ii) that
		\[
		f_i(\IN(G_{n+1}))=f_i(\IN(G_n))
		\quad\text{for every }i\ge s.
		\]
		This implies $\alpha(G_{n+1})=\alpha(G_n)$, as desired. 
		
		Set $d=\alpha(G_n)$. Writing the Hilbert series in terms of the $h$-vector as in \eqref{eq:h-vector}, we have
		\[
		\HS_{G_n}(t)
		=\frac{\sum_{i=0}^d h_i(\IN(G_n))t^i}{(1-t)^d}.
		\]
		Substituting this expression into \eqref{eq:HS-difference} gives
		\[
		\sum_{i=0}^d
		\bigl(h_i(\IN(G_{n+1}))-h_i(\IN(G_n))\bigr)t^i
		=t(1-t)^{d-s}.
		\]
		Expanding $t(1-t)^{d-s}$ and comparing coefficients, we get
		\[
		h_i(\IN(G_{n+1})) - h_i(\IN(G_n))=	(-1)^{i-1}\binom{d-s}{i-1} \quad \text{for all } i\ge 0.
		\]
		This proves (iii).
		
		Finally, (iv) follows immediately from (iii) and the fact that 
		\[
		\chi(\IN(G_n))-1
		= \widetilde{\chi}(\IN(G_n))
		=(-1)^{d-1}h_d(\IN(G_n)).
		\qedhere
		\]
	\end{proof}
	
	We conclude the section by studying the asymptotic behavior of the number of facets of $\IN(G_n)$. The data in \Cref{tab:independence} suggest that this number is eventually linear in $n$. The next result confirms this observation.
	
	\begin{thm}
		\label{conj:IN-facet}
		Let $\Gc=(G_n)_{n\ge1}$ be an $\Inc$-invariant chain of graphs with $\ind(\Gc)=r$ and $\spi(\Gc)=s$. Denote by $\nu(\IN(G_n))$ the number of facets of $\IN(G_n)$. Then there exists a constant $C\in\Z$ such that
		\[
		\nu(\IN(G_n))=n+C
		\quad\text{for all }n\ge 5r+2s.
		\]
	\end{thm}
	
	The proof relies on the observation that every facet of $\IN(G_n)$ is either a central interval or is contained in a boundary region of bounded size. More precisely, we show that a facet is either an interval contained in $[2r,n-2r]$ or a facet of the subgraph induced by
	\[
	[1,2r+s-2]\cup[n-2r-s+2,n].
	\]
	For brevity, set
	\[
	M_1=[r,n-r]
	\quad\text{and}\quad
	M_2=[2r,n-2r].
	\]
	Recall that the set of facets of a simplicial complex $\Delta$ is denoted by $\FF(\Delta)$.
	
	\begin{lem}
		\label{lem:I_t}
		Keep the assumptions of \Cref{conj:IN-facet}. For $t\ge1$, let $I_t=[t,t+s-1]$. If $I_t\subseteq M_2$, then $I_t\in\FF(\IN(G_n))$.
	\end{lem}
	
	\begin{proof}
		By \Cref{lem:independent-interval}(ii), $I_t$ is a face of $\IN(G_n)$. It remains to show that $I_t\cup\{u\}$ is not independent for every $u\in[n]\setminus I_t$. We distinguish two cases.
		
		\smallskip
		\emph{Case 1: $u<t$.}
		Then $u\le t-1<t+s-1$. Since $I_t\subseteq M_2$, both $t-1$ and $t+s-1$ belong to $M_1$. Hence,
		$(t-1,\,t+s-1)\in E(G_n)$
		by \Cref{lem:independent-interval}(i). If $u\ge r$, a west move (\Cref{lem.shortmoves}(ii)) applied to this edge gives $(u,t+s-1)\in E(G_n)$.
		Now suppose that $u<r$. By \Cref{assm:main}, $(1,i)\in E(G_r)$ for some $i\in[r]$. Since $t\in M_2$, we have
		\[
		0\le u-1\le t-i\le n-r.
		\]
		Thus, $(u,t)\in E(G_n)$ by \Cref{lem:E(Gn)}.
		
		\smallskip
		\emph{Case 2: $u>t+s-1$.}
		Then $u-t\ge s$. If $u\le n-r$, \Cref{lem:independent-interval}(i) gives $(t,u)\in E(G_n)$. Suppose that $u>n-r$. By \Cref{assm:main}, there exists $(b,r)\in E(G_r)$ for some $b\in[r]$. Since $t\in M_2$, we have $t-b<n-2r<u-r$. Hence,
		\[
		0<t-b<u-r\le n-r,
		\]
		and thus $(t,u)\in E(G_n)$ by \Cref{lem:E(Gn)}.
		
		In both cases, $I_t\cup\{u\}$ is not independent. Therefore, $I_t$ is a facet of $\IN(G_n)$.
	\end{proof}
	
	\begin{lem}
		\label{lem:IN-facet-dec}
		Keep the assumptions and notation of \Cref{lem:I_t}. For $n\ge4r+3s$, set
		\[
		U_n=[2r+s-1,n-2r-s+1].
		\]
		Then $\FF(\IN(G_n))$ admits the following disjoint decomposition:
		\begin{equation}
			\label{eq:facet-dec}
			\FF(\IN(G_n))
			=\{I_t\mid I_t\subseteq M_2\}
			\sqcup\FF(\IN(G_n\setminus U_n)).
		\end{equation}
	\end{lem}
	
	\begin{proof}
		Let $n\ge4r+3s$. We first establish several claims.
		
		\begin{claim}
			\label{cl:M1-M2}
			Suppose that $u\in[n]\setminus M_1$ and $v\in M_2$. If $\{u,v\}\in E(G_n)$, then $\{u,w\}\in E(G_n)$ for every $w\in M_2$.
		\end{claim}
		
		\begin{proof}[Proof of \Cref{cl:M1-M2}]
			Since $u\notin M_1$, either $u<r$ or $u>n-r$. We consider only the case $u<r$, as the other case is similar. Since $(u,v)\in E(G_n)$, \Cref{lem:E(Gn)} yields an edge $(i,j)\in E(G_r)$ such that
			\[
			0\le u-i\le v-j\le n-r.
			\]
			For every $w\in M_2$, the above inequalities are evidently still valid if $v$ is replaced by $w$. Hence, $(u,w)\in E(G_n)$ by \Cref{lem:E(Gn)}.
		\end{proof}
		
		\begin{claim}
			\label{cl:interval}
			For every $F\in\FF(\IN(G_n))$, the set $F\cap M_1$ is an interval.
		\end{claim}
		
		\begin{proof}[Proof of \Cref{cl:interval}]
			Set $A=F\cap M_1$, and let $x,z\in A$ with $x<z$. We show that every $y$ with $x<y<z$ belongs to $A$. Suppose otherwise. Since $y\in M_1$ and $y\notin A$, we have $y\notin F$. The maximality of $F$ yields a vertex $w\in F$ such that $\{w,y\}\in E(G_n)$.
			
			If $w<y$, then $w<y<z\le n-r$, and a north move applied to $(w,y)$ gives $(w,z)\in E(G_n)$ by \Cref{lem.shortmoves}(iii), contradicting the independence of $F$. If $y<w$, then $r\le x<y<w$, and a west move applied to $(y,w)$ gives $(x,w)\in E(G_n)$ by \Cref{lem.shortmoves}(ii), again a contradiction. Thus, $A$ is an interval.
		\end{proof}
		
		\begin{claim}
			\label{cl:Un}
			Let $F\in\FF(\IN(G_n))$. If $F\cap U_n\ne\emptyset$, then $F=I_t$ for some interval $I_t\subseteq M_2$.
		\end{claim}
		
		\begin{proof}[Proof of \Cref{cl:Un}]
			Set $A=F\cap M_1$. By \Cref{cl:interval}, $A$ is an interval. Moreover, since $A$ is an independent set in $M_1$, \Cref{lem:independent-interval}(i) implies that $|A|\le s$. As $U_n\subseteq M_1$, we have
			\[
			A\cap U_n=F\cap U_n\ne\emptyset.
			\]
			Choose $u\in A\cap U_n$ and set $t=\min A$. Since $|A|\le s$ and $A$ is an interval,
			\[
			t\ge u-s+1\ge2r.
			\]
			Furthermore, $t\le u\le n-2r-s+1$, and hence $t+s-1\le n-2r$. Therefore,
			\[
			A\subseteq I_t\subseteq M_2.
			\]
			Suppose that $A\ne I_t$, and choose $v\in I_t\setminus A$. Since $I_t\subseteq M_2\subseteq M_1$, we have $v\notin F$. Thus, $F\cup\{v\}$ is not independent, and we can find $x\in F$ such that $\{x,v\}\in E(G_n)$. Because $I_t$ is independent,
			\[
			x\in F\setminus I_t
			\subseteq F\setminus A
			=F\setminus M_1
			\subseteq[n]\setminus M_1.
			\]
			Since $v\in I_t\subseteq M_2$, \Cref{cl:M1-M2} implies that $x$ is adjacent to every vertex of $M_2$, and hence to every vertex of $A$. This contradicts the independence of $F$. Therefore, $A=I_t$. By \Cref{lem:I_t}, $I_t$ is a facet of $\IN(G_n)$, so $F=I_t$.
		\end{proof}
		
		\begin{claim}
			\label{cl:intersection}
			One has
			\[
			\{I_t\mid I_t\subseteq M_2\}
			\cap\FF(\IN(G_n\setminus U_n))=\emptyset.
			\]
		\end{claim}
		
		\begin{proof}[Proof of \Cref{cl:intersection}]
			Suppose that $I_t$ belongs to the intersection. Since $I_t$ is disjoint from $U_n$,
			\[
			I_t\subseteq M_2\cap([n]\setminus U_n)
			=[2r,2r+s-2]\sqcup[n-2r-s+2,n-2r].
			\]
			As $I_t$ is an interval, it must be contained in one of the two intervals on the right. This is impossible because each has $s-1$ vertices, whereas $|I_t|=s$.
		\end{proof}
		
		We now prove \eqref{eq:facet-dec}. By \Cref{cl:intersection}, the union on the right-hand side is disjoint. The inclusion ``$\subseteq$'' follows from \Cref{cl:Un}. For the reverse inclusion, \Cref{lem:I_t} shows that it remains to prove
		\[
		\FF(\IN(G_n\setminus U_n))\subseteq\FF(\IN(G_n)).
		\]
		Let $F\in\FF(\IN(G_n\setminus U_n))$. Then $F$ is an independent set of $G_n$. Suppose that $F$ is not a facet of $\IN(G_n)$, and choose $F'\in\FF(\IN(G_n))$ with $F\subsetneq F'$. Since $F$ is maximal among the independent sets contained in $[n]\setminus U_n$, we must have $F'\cap U_n\ne\emptyset$. By \Cref{cl:Un}, $F'=I_t$ for some $I_t\subseteq M_2$. The maximality of $F$ in $G_n\setminus U_n$ now gives
		\[
		F=F'\cap([n]\setminus U_n)
		=I_t\cap\bigl([1,2r+s-2]\sqcup[n-2r-s+2,n]\bigr)
		=F_1\sqcup F_2,
		\]
		where
		\[
		F_1=I_t\cap[1,2r+s-2]
		\quad\text{and}\quad
		F_2=I_t\cap[n-2r-s+2,n].
		\]
		The distance between the two boundary intervals is greater than $s$, since
		\[
		n-2r-s+2-(2r+s-2)=n-4r-2s+4>s.
		\]
		As $|I_t|=s$, at least one of $F_1$ and $F_2$ is empty. Thus, $F$ is an interval. Since $|F|<|F'|=s$, we may enlarge $F$ to an interval $I_{t'}\subseteq[n]\setminus U_n$ of cardinality $s$. By \Cref{lem:independent-interval}(ii), $I_{t'}$ is independent in $G_n\setminus U_n$, contradicting the maximality of $F$. This proves the reverse inclusion and hence the decomposition.
	\end{proof}
	
	We can now prove \Cref{conj:IN-facet}.
	
	\begin{proof}[Proof of \Cref{conj:IN-facet}]
		Let $n\ge 5r+2s$. Then $n\ge 4r+3s$ since $r\ge s$. By \Cref{lem:IN-facet-dec},
		\begin{align*}
			\nu(\IN(G_n))
			&=|\FF(\IN(G_n))|
			=|\{I_t\mid I_t\subseteq M_2\}|
			+|\FF(\IN(G_n\setminus U_n))|\\
			&=|\{I_t\mid t\in[2r,n-2r-s+1]\}|
			+|\FF(\IN(G_n\setminus U_n))|\\
			&=n-4r-s+2+|\FF(\IN(G_n\setminus U_n))|.
		\end{align*}
		Set
		$\alpha=2r+s-1$
		and
		$\beta=n-2r-s+1$.
		Then $U_n=[\alpha,\beta]$ and $U_{n+1}=[\alpha,\beta+1]$.
		Moreover, $\beta-\alpha> r$ since $n\ge5r+2s$.
		Hence,
		$
		G_n\setminus U_n\cong G_{n+1}\setminus U_{n+1}
		$
		by \Cref{lem:del-graph-iso}. It follows that 
		\[
		|\FF(\IN(G_n\setminus U_n))|
		=|\FF(\IN(G_{n+1}\setminus U_{n+1}))|.
		\]
		Consequently,
		\[
		\nu(\IN(G_{n+1}))=\nu(\IN(G_n))+1
		\quad\text{for all }n\ge 5r+2s.
		\]
		This immediately yields the desired assertion.
	\end{proof}

	%%%%%%%%%%%%%%%%%%%%%%%%%%%%%%%%%%%%%%%%%%%%%%%%%%%%%%%%%%%%
	
	\section{Clique complexes}\label{subsec.cliq}
	
	This section is devoted to the study of the asymptotic behavior of clique complexes along an $\Inc$-invariant chain of graphs. Our main result establishes recurrence relations for the $f$-vectors and $h$-vectors of these complexes, from which we determine their asymptotic growth. We also show that the clique number is eventually quasi-linear.
	
	We begin by recalling the notion of a quasi-polynomial. A function $\psi:\N\to\Q$ is called \emph{quasi-polynomial} if there exist an integer $N\in\N$ and polynomials $P_0,P_1,\dots,P_{N-1}\in\Q[t]$ such that
	\[
	\psi(n)=P_k(n)
	\quad\text{if } n\equiv k \pmod N.
	\]
	The integer $N$ is called a \emph{quasiperiod} of $\psi$, and the smallest quasiperiod is called its \emph{period}. Thus, $\psi$ is a polynomial function precisely when it has period $1$. The \emph{degree} of $\psi$ is
	\[
	\deg\psi=\max\{\deg P_k\mid0\le k\le N-1\}.
	\]
	A quasi-polynomial of degree $1$ is called \emph{quasi-linear}.
	
	\begin{thm}
		\label{thm:clique}
		Let $\Gc=(G_n)_{n\ge1}$ be an $\Inc$-invariant chain of graphs with $\ind(\Gc)=r$ and $\spi(\Gc)=s$. Denote by $\Cl(G_n)$ and $\omega(G_n)$ the clique complex and the clique number of $G_n$, respectively. Let
		\[
		(f_{-1}(\Cl(G_n)),f_0(\Cl(G_n)),f_1(\Cl(G_n)),\dots)
		\quad\text{and}\quad
		(h_0(\Cl(G_n)),h_1(\Cl(G_n)),\dots)
		\] 
		be the $f$-vector and the $h$-vector of $\Cl(G_n)$, respectively. Then there exists $n_0\ge r+s$ such that, for every $n\ge n_0$, the following statements hold:
		\begin{enumerate}
			\item The function $\omega(G_n)$ is eventually quasi-linear of period $s$. More precisely,
			\[
			\omega(G_n)
			=\omega(G_{n_0})+\left\lfloor\frac{n-n_0}{s}\right\rfloor.
			\]
			\item For every $i\ge0$,
			\[
			f_i(\Cl(G_n))
			=f_i(\Cl(G_{n-1}))+f_{i-1}(\Cl(G_{n-s})).
			\]
			\item With the convention $h_{-1}(\Cl(G_k))=0$, one has, for every $i\ge0$,
			\[
			h_i(\Cl(G_n))=
			\begin{cases}
				h_i(\Cl(G_{n-1}))+h_{i-1}(\Cl(G_{n-s}))&\text{if } s\nmid(n-n_0),\\
				h_i(\Cl(G_{n-1}))-h_{i-1}(\Cl(G_{n-1}))+h_{i-1}(\Cl(G_{n-s}))&\text{if } s\mid(n-n_0).
			\end{cases}
			\]
			\item The ordinary and reduced Euler characteristics satisfy
			\begin{align*}
				\chi(\Cl(G_n))
				&=\chi(\Cl(G_{n-1}))-\chi(\Cl(G_{n-s}))+1,\\
				\widetilde{\chi}(\Cl(G_n))
				&=\widetilde{\chi}(\Cl(G_{n-1}))-\widetilde{\chi}(\Cl(G_{n-s})).
			\end{align*}
		\end{enumerate}
	\end{thm}
	
	We illustrate \Cref{thm:clique} with the following example.
	
	\begin{ex}
		\label{ex:clique}
		Let $\mathcal{G}=(G_n)_{n\ge1}$ be an $\Inc$-invariant chain of graphs with $\ind(\mathcal{G})=5$ and
		\[
		E(G_5)=\{(1,4),(2,4),(3,5)\}.
		\]
		Then $\spi(\Gc)=2$. Computations with Macaulay2 \cite{GS} yield the following table.
		
		\begin{table}[htbp]
			\centering
			{\small
				\setlength{\tabcolsep}{4pt}
				\begin{tabular}{|c|c|c|c|c|c|}
					\hline
					$n$ & $\omega(G_n)$& $\gamma(G_n)$ & $f$-vector of $\Cl(G_n)$ & $h$-vector of $\Cl(G_n)$  & $\widetilde{\chi}(\Cl(G_n))$ \\
					\hline
					5  & 2 & 2& $(1,5,3)$          & $(1,3,-1)$       & $1$  \\
					\hline
					6  & 2 & 3 & $(1,6,7)$        & $(1,4,2)$    & $-2$ \\
					\hline
					7  & 3 & 3 & $(1,7,12,3)$       & $(1,4,1,-3)$     & $-3$ \\
					\hline
					8  & 3 & 4 & $(1,8,18,10)$    & $(1,5,5,-1)$   & $-1$  \\
					\hline
					9  & 4 & 4 & $(1,9,25,22,3)$    & $(1,5,4,-5,-2)$  & $2$  \\
					\hline
					10 & 4 & 5 & $(1,10,33,40,13)$ & $(1,6,9,0,-3)$  & $3$  \\
					\hline
			\end{tabular}}
			\caption{Some invariants of clique complexes}
			\label{tab:clique}
		\end{table}
		
		The reader may verify that the conclusions of \Cref{thm:clique} hold throughout the range $7\le n\le10$, with $n_0=7$. For example,
		\begin{align*}
			f_2(\Cl(G_{10}))
			&=40=22+18
			=f_2(\Cl(G_9))+f_1(\Cl(G_8)),\\
			h_3(\Cl(G_7))
			&=-3=0-2+(-1)
			=h_3(\Cl(G_6))-h_2(\Cl(G_6))+h_2(\Cl(G_5)),\\
			h_1(\Cl(G_8))
			&=5=4+1
			=h_1(\Cl(G_7))+h_0(\Cl(G_6)).
		\end{align*}
		
		The third column of the table records the chromatic number of $G_n$, which will be studied in the next section.
	\end{ex}
	
	The following consequence of \Cref{thm:clique} describes the asymptotic growth of the $f$-vector and the $h$-vector of the clique complex. 
	
	\begin{cor}
		\label{cor:f-h-clique}
		Under the assumptions of \Cref{thm:clique}, the following statements hold:
		\begin{enumerate}
			\item For each $i\ge-1$, $f_i(\Cl(G_n))$ is eventually a polynomial in $n$ of degree $i+1$ with leading coefficient $1/(i+1)!$. More precisely,
			\[
			f_i(\Cl(G_n))
			=\frac{1}{(i+1)!}n^{i+1}+O(n^i)
			\quad\text{as }n\longrightarrow\infty.
			\]
			\item For each $i\ge0$, $h_i(\Cl(G_n))$ is eventually a quasi-polynomial in $n$ of degree at most $i$ and quasiperiod $s$. More precisely,
			\[
			h_i(\Cl(G_n))
			=\frac{1}{i!}\left(1-\frac{1}{s}\right)^i n^i
			+O(n^{i-1})
			\quad\text{as }n\longrightarrow\infty.
			\]
			In particular, if $s>1$, then $h_i(\Cl(G_n))$ has degree exactly $i$.
		\end{enumerate}
	\end{cor}
	
	\begin{proof}
		Set $\Delta_n=\Cl(G_n)$.
		
		(i) We proceed by induction on $i$. The case $i=-1$ is immediate, since $f_{-1}(\Delta_n)=1$ for every $n\ge1$. Let $i\ge0$, and assume that $f_{i-1}(\Delta_n)$ is eventually a polynomial in $n$ of degree $i$ with leading coefficient $1/i!$. Thus, there exist a polynomial $P\in\Q[t]$ of degree $i$ with leading coefficient $1/i!$ and an integer $n_0$ such that
		\[
		f_{i-1}(\Delta_n)=P(n)
		\quad\text{for all }n\ge n_0.
		\]
		By \Cref{thm:clique}, there exists $n_1\ge n_0+s$ such that
		\[
		f_i(\Delta_n)
		=f_i(\Delta_{n-1})+f_{i-1}(\Delta_{n-s})
		\quad\text{for all }n\ge n_1.
		\]
		Therefore,
		\[
		f_i(\Delta_n)
		=f_i(\Delta_{n_1-1})
		+\sum_{k=n_1}^{n}f_{i-1}(\Delta_{k-s})
		=f_i(\Delta_{n_1-1})
		+\sum_{k=n_1}^{n}P(k-s)
		\]
		for all $n\ge n_1$. Since the last sum is a polynomial in $n$ of degree $i+1$ with leading coefficient $1/(i+1)!$ (see, e.g., \cite[Section~2.6]{GKP}), the assertion follows.
		
		(ii) Set $d_n=\omega(G_n)$. By \Cref{thm:clique}(i), $d_n$ is eventually quasi-linear of period $s$ and
		\[
		d_n=\frac{n}{s}+O(1)
		\quad\text{as }n\longrightarrow\infty.
		\]
		For every $k\in\{0,\dots,i\}$, it follows that
		\[
		\binom{d_n-k}{i-k}
		=\frac{1}{(i-k)!}\left(\frac{n}{s}\right)^{i-k}
		+O(n^{i-k-1})
		\quad\text{as }n\longrightarrow\infty,
		\]
		and this is eventually a quasi-polynomial of degree $i-k$ and quasiperiod $s$. By part (i),
		\[
		f_{k-1}(\Delta_n)
		=\frac{1}{k!}n^k+O(n^{k-1})
		\quad\text{as }n\longrightarrow\infty.
		\]
		Using \eqref{eq:f-and-h-vector}, we obtain
		\[
		h_i(\Delta_n)
		=\sum_{k=0}^i(-1)^{i-k}
		\binom{d_n-k}{i-k}f_{k-1}(\Delta_n).
		\]
		Thus, $h_i(\Delta_n)$ is eventually a quasi-polynomial of quasiperiod $s$ and degree at most $i$. Its coefficient of $n^i$ is
		\begin{align*}
			\sum_{k=0}^i
			(-1)^{i-k}\frac{1}{(i-k)!k!s^{i-k}}
			&=\frac{1}{i!}\sum_{k=0}^i
			\binom{i}{k}\left(-\frac{1}{s}\right)^{i-k}
			=\frac{1}{i!}\left(1-\frac{1}{s}\right)^i.
		\end{align*}
		This concludes the proof.
	\end{proof}
	
	We now turn to the proof of \Cref{thm:clique}. As in the proof of \Cref{thm:f-vector-IN}, we exploit the relations among the $f$-vector, the $h$-vector, and the Hilbert series given in \eqref{eq:f-vector} and \eqref{eq:h-vector}. Since $\Cl(G_n)=\IN(G_n^c)$, parts (ii)--(iv) of \Cref{thm:clique} will follow from the next proposition.
	
	\begin{prop}
		\label{pro:Hilbert-Cl}
		Let $\Gc=(G_n)_{n\ge1}$ be an $\Inc$-invariant chain of graphs with $\ind(\Gc)=r$ and $\spi(\Gc)=s$. Then
		\[
		\HS_{G_n^c}(t)
		=\HS_{G_{n-1}^c}(t)
		+\frac{t}{1-t}\HS_{G_{n-s}^c}(t)
		\quad\text{for all }n\gg0.
		\]
	\end{prop}
	
	As before, we first prove the assertion for eventually saturated chains. Recall that $N_G(v)$ denotes the open neighborhood of a vertex $v$ in a graph $G$.
	
	\begin{lem}
		\label{lem:clique-saturarated}
		Let $\Gc=(G_n)_{n\ge1}$ be an eventually saturated $\Inc$-invariant chain of graphs with $\ind(\Gc)=r$ and $\spi(\Gc)=s$. Then
		\[
		\HS_{G_n^c}(t) = \HS_{G_{n-1}^c}(t)+\frac{t}{1-t}\HS_{G_{n-s}^c}(t)
		\quad\text{for all }n\ge 2r+1.
		\]
	\end{lem}
	
	\begin{proof}
		By \Cref{saturated}(i),
		\[
		V\bigl(G_n\setminus N_{G_n}[n]\bigr)=[n-s+1,n-1]
		\quad\text{for all }n\ge2r+1.
		\]
		Hence, $N_{G_n}(n)=[n-s]$, and therefore
		\[
		N_{G_n^c}[n]=[n-s+1,n]
		\quad\text{for all }n\ge 2r+1.
		\]
		Since $\Gc$ is eventually saturated, \Cref{lem:saturation-characterization} gives $G_n\setminus n=G_{n-1}$, whence
		\[
		G_n^c\setminus n=G_{n-1}^c
		\quad \text{for all } n\ge r+1.
		\]
		Applying this equality successively yields
		\[
		G_n^c\setminus N_{G_n^c}[n]
		=G_n^c\setminus[n-s+1,n]
		=G_{n-s}^c
		\quad\text{for all }n\ge2r+1.
		\]
		The assertion now follows from \Cref{lem:Hilbert-series-graph} applied to $G_n^c$ and the vertex $n$.
	\end{proof}
	
	Let us now prove \Cref{pro:Hilbert-Cl}.
	
	\begin{proof}[Proof of \Cref{pro:Hilbert-Cl}]
		We proceed by induction on $\me(\Gc)$. If $\Gc$ is eventually saturated, the assertion follows from \Cref{lem:clique-saturarated}. We may therefore assume that $\Gc$ is not eventually saturated and, by induction, that the assertion holds for every chain $\Hc$ with $\me(\Hc)<\me(\Gc)$.
		For $n\ge2$,
		\[
		G_n\setminus\{n\}=G_n\cap K_{n-1}
		\quad\text{and}\quad
		G_n\setminus N_{G_n^c}[n]
		=G_n\cap K_{N_{G_n}(n)},
		\]
		where the second equality follows from the partition
		$V(G_n)=N_{G_n}(n)\sqcup N_{G_n^c}[n].$
		This partition together with \Cref{lem.GWeaklyChordal}(i) yields
		\[
		N_{G_n}(n)=[b(\Gc),n-r+B(\Gc)]
		\quad\text{for all }n\ge2r+1.
		\]
		Set
		$\alpha=b(\Gc)$, $\beta=r-B(\Gc)$, and $U_n=[\alpha,n-\beta]$. Then $U_n=N_{G_n}(n)$ for $n\ge2r+1.$ 
		Define the chains $\Hc=(H_n)_{n\ge1}$ and $\Fc=(F_n)_{n\ge1}$ by
		\[
		H_n=
		\begin{cases}
			\emptyset      & \text{if }1\le n\le r-1,\\
			G_{n+1}\cap K_n & \text{if }n\ge r
		\end{cases}
		\]
		and
		\[
		F_n=
		\begin{cases}
			\emptyset & \text{if }1\le n\le r+\alpha-2,\\
			G_{n+\beta}\cap K_{U_{n+\beta}} & \text{if }n\ge r+\alpha-1.
		\end{cases}
		\]
		By \Cref{lem:induced-chain}, both $\Hc$ and $\Fc$ are $\Inc$-invariant chains satisfying
		$\spi(\Hc)=\spi(\Fc)=s.$
		Moreover, $\beta=r-B(\Gc)\ge1$, and since $\Gc$ is not eventually saturated, \Cref{lem:induced-chain}(iii) gives
		\[
		\me(\Hc)<\me(\Gc)
		\quad\text{and}\quad
		\me(\Fc)<\me(\Gc).
		\]
		The induction hypothesis applied to $\Hc$ and $\Fc$ yields
		\begin{align*}
			\HS_{H_n^c}(t)
			&=\HS_{H_{n-1}^c}(t)
			+\frac{t}{1-t}\HS_{H_{n-s}^c}(t),\\
			\HS_{F_n^c}(t)
			&=\HS_{F_{n-1}^c}(t)
			+\frac{t}{1-t}\HS_{F_{n-s}^c}(t)
		\end{align*}
		for all $n\gg0$. 
		Note that
		\[
		F_{n-\beta}
		=G_n\cap K_{U_n}
		=G_n\cap K_{N_{G_n}(n)}
		=G_n\setminus N_{G_n^c}[n]
		\]
		for all $n\ge2r+1$. Thus, applying \Cref{lem:Hilbert-series-graph} to $G_n^c$ at the vertex $n$ gives
		\begin{align*}
			\HS_{G_n^c}(t) &= \HS_{H_{n-1}^c}(t)+\frac{t}{1-t}\HS_{F_{n-\beta}^c}(t)\\
			&=\HS_{H_{n-2}^c}(t)+\frac{t}{1-t}\HS_{H_{n-s-1}^c}(t)+
			\frac{t}{1-t}\Big(\HS_{F_{n-\beta-1}^c}(t)+\frac{t}{1-t}\HS_{F_{n-\beta-s}^c}(t)\Big)\\
			&=\HS_{H_{n-2}^c}(t)+\frac{t}{1-t}\HS_{F_{n-\beta-1}^c}(t)
			+\frac{t}{1-t}\Big(\HS_{H_{n-s-1}^c}(t)+\frac{t}{1-t}\HS_{F_{n-\beta-s}^c}(t)\Big)\\
			&= \HS_{G_{n-1}^c}(t)+\frac{t}{1-t}\HS_{G_{n-s}^c}(t)
		\end{align*}
		for all $n\gg0$. This completes the induction.
	\end{proof}
	
	The proof of \Cref{thm:clique}(i) requires further preparations. Let $s\in\N$. A function $\psi\colon\N\to\Z$ is called \emph{eventually quasi-linear with slope $1/s$} if there exist $p,q\in\Z$ such that
	\[
	\psi(n)=p+\left\lfloor\frac{n+q}{s}\right\rfloor
	\quad\text{for all }n\gg0.
	\]
	For such a function, one may choose $n_0$ so that
	\[
	\psi(n)
	=\psi(n_0)+\left\lfloor\frac{n-n_0}{s}\right\rfloor
	\quad\text{for all }n\ge n_0.
	\]
	Indeed, choose $n_0$ sufficiently large that the preceding formula for $\psi$ holds for all $n\ge n_0-s$ and $n_0+q$ is divisible by $s$. If $n_0+q=ks$, then $\psi(n_0)=p+k$, and the desired identity follows. With this choice, $\psi$ is nondecreasing and
	\[
	\psi(n)=\psi(n-s)+1
	\quad\text{for all }n\ge n_0.
	\]
	
	The next proposition characterizes these functions.
	
	\begin{prop}
		\label{lem:quasi-linear}
		Let $\psi\colon\N\to\Z$ be a function, and let $s\in\N$. The following are equivalent:
		\begin{enumerate}
			\item $\psi$ is eventually quasi-linear with slope $1/s$;
			\item $\psi(n-1)\le\psi(n)=\psi(n-s)+1$ for all $n\gg0$;
			\item $\psi(n)=\max\{\psi(n-1),\psi(n-s)+1\}$ for all $n\gg0$;
			\item $\psi(n-1)\le\psi(n)\le\psi(n-s)+1$ and $\psi(n)\ge\varphi(n)$ for all $n\gg0$, where $\varphi$ is eventually quasi-linear with slope $1/s$;
			\item $\psi(n)\ge\max\{\psi(n-1),\psi(n-s)+1\}$ and $\psi(n)\le\phi(n)$ for all $n\gg0$, where $\phi$ is eventually quasi-linear with slope $1/s$.
		\end{enumerate}
	\end{prop}
	
	\begin{proof}
		The implications (i)$\Rightarrow$(ii)$\Rightarrow$(iii) are immediate.
		
		(iii)$\Rightarrow$(ii): Assume that the recurrence in (iii) holds for every $n\ge k$. It follows immediately that $\psi$ is nondecreasing on $[k,\infty)$. Since $\psi(n-s)\le\psi(n-1)$, the recurrence becomes
		\begin{equation}
			\label{eq:psi(n)}
			\psi(n)=
			\begin{cases}
				\psi(n-1)   & \text{if }\psi(n-s)<\psi(n-1),\\
				\psi(n-s)+1 & \text{if }\psi(n-s)=\psi(n-1).
			\end{cases}
		\end{equation}
		We claim that $\psi(n)=\psi(n-s)+1$ for all sufficiently large $n$. Indeed, suppose that
		\[
		\psi(a)>\psi(a-s)+1
		\quad\text{for some } a\ge k+s.
		\]
		Then \eqref{eq:psi(n)} gives $\psi(a)=\psi(a-1)$. Let $t\in\{1,\dots,s-1\}$ be maximal such that $\psi(a)=\psi(a-t)$. For $1\le j\le s-t-1$, monotonicity gives
		\[
		\psi(a+j-1)\ge\psi(a)>\psi(a-t-1)\ge\psi(a+j-s),
		\]
		and hence \eqref{eq:psi(n)} yields $\psi(a+j)=\psi(a+j-1)$. Setting $\ell=a-t$, we obtain
		\begin{equation}
			\label{eq:psi-constant}
			\psi(\ell)=\psi(\ell+1)=\cdots=\psi(\ell+s-1).
		\end{equation}
		By \eqref{eq:psi(n)}, this implies $\psi(\ell+s)=\psi(\ell)+1$. Since $\psi(\ell+s)>\psi(\ell)=\psi(\ell+1)$, it follows again from  \eqref{eq:psi(n)} that $\psi(\ell+s+1)=\psi(\ell+s)=\psi(\ell)+1$.
		Repeating this argument in combination with \eqref{eq:psi-constant}, we obtain
		\[
		\psi(\ell+ps+q)=\psi(\ell)+p
		\quad\text{for all }p\ge0\text{ and }0\le q\le s-1.
		\]
		Consequently, $\psi(n)=\psi(n-s)+1$ for all $n\ge\ell+s$, proving (ii).
		
		(ii)$\Rightarrow$(i): Choose $k$ such that
		\[
		\psi(n-1)\le\psi(n)=\psi(n-s)+1
		\quad\text{for all }n\ge k.
		\]
		Define
		\[
		\delta(n)=\psi(n+1)-\psi(n).
		\]
		For every $n\ge k$, we have
		\[
		0\le\delta(n)\le\psi(n+s)-\psi(n)=1
		\]
		and
		\begin{equation}
			\label{eq:periodicity}
			\delta(n+s)=\delta(n).
		\end{equation}
		Moreover,
		\[
		1=\psi(k+s)-\psi(k)=\sum_{i=k}^{k+s-1}\delta(i).
		\]
		Thus, there is a unique $j\in[k,k+s-1]$ such that $\delta(j)=1$ and $\delta(i)=0$ for $i\in[k,k+s-1]\setminus\{j\}$. By \eqref{eq:periodicity},
		\[
		\delta(n)=
		\begin{cases}
			1 & \text{if }n\equiv j\pmod s,\\
			0 & \text{otherwise}
		\end{cases}
		\]
		for all $n\ge k$. Set $n_0=j+1$.
		Then
		\[
		\psi(n)
		=\psi(n_0)+\sum_{i=n_0}^{n-1}\delta(i)
		=\psi(n_0)+\left\lfloor\frac{n-n_0}{s}\right\rfloor
		\quad\text{for all }n\ge n_0.
		\]
		Hence, $\psi$ is eventually quasi-linear with slope $1/s$.
		
		(i)$\Rightarrow$(iv), (v): Take $\varphi=\psi$ and $\phi=\psi$.
		
		(iv)$\Rightarrow$(ii): Set $\tau=\psi-\varphi$. Since $\varphi(n)=\varphi(n-s)+1$ for $n\gg0$, it follows from the assumptions on $\psi$ that $\tau(n)\in\Z_{\ge0}$ and
		$\tau(n)\le\tau(n-s)$
		for all $n\gg0$. This implies $\tau(n)=\tau(n-s)$ for all $n\gg0$. Hence,
		\[
		\psi(n)-\psi(n-s)
		=\varphi(n)-\varphi(n-s)=1
		\quad\text{for all }n\gg0.
		\]
		Together with the assumed inequality $\psi(n-1)\le\psi(n)$, this proves (ii).
		
		(v)$\Rightarrow$(ii): Setting $\theta=\phi-\psi$, the proof proceeds analogously to that of the implication (iv)$\Rightarrow$(ii). 
		The detail is left to the reader.
	\end{proof}
	
	We are now in a position to prove \Cref{thm:clique}.
	
	\begin{proof}[Proof of \Cref{thm:clique}]
		Set $\Delta_n=\Cl(G_n)$. Since $\Delta_n=\IN(G_n^c)$, \eqref{eq:f-vector} gives
		\[
		\HS_{G_n^c}(t)
		=\sum_{i\ge-1}
		f_i(\Delta_n)\frac{t^{i+1}}{(1-t)^{i+1}}.
		\]
		Combining this identity with \Cref{pro:Hilbert-Cl} and comparing the coefficients of the powers of $t/(1-t)$, we obtain
		\[
		f_i(\Delta_n)
		=f_i(\Delta_{n-1})+f_{i-1}(\Delta_{n-s})
		\quad\text{for every }i\ge0\text{ and all }n\gg0.
		\]
		This proves (ii).
		
		Set $d_n=\omega(G_n)$. The preceding recurrence yields
		\[
		d_n=\max\{d_{n-1},d_{n-s}+1\}
		\quad\text{for all }n\gg0.
		\]
		Applying \Cref{lem:quasi-linear}, we obtain (i). In particular, we can choose $n_0$ sufficiently large such that 
		\[
		d_n=d_{n_0}+\left\lfloor\frac{n-n_0}{s}\right\rfloor
		\quad\text{for all }n\ge n_0.
		\]
		This implies
		\[
		d_n-d_{n-1}=
		\begin{cases}
			0 & \text{if }s\nmid(n-n_0),\\
			1 & \text{if }s\mid(n-n_0),
		\end{cases}
		\]
		and moreover, $d_n=d_{n-s}+1$ for all $n\ge n_0$.
		
		Let $h(\Delta;t)$ denote the $h$-polynomial of a simplicial complex $\Delta$. Multiplying the Hilbert series recurrence in \Cref{pro:Hilbert-Cl} by $(1-t)^{d_n}$ gives
		\[
		h(\Delta_n;t)
		=(1-t)^{d_n-d_{n-1}}h(\Delta_{n-1};t)
		+t h(\Delta_{n-s};t).
		\]
		Therefore,
		\begin{equation}
			\label{eq:clique-h}
			h(\Delta_n;t)=
			\begin{cases}
				h(\Delta_{n-1};t)+t h(\Delta_{n-s};t)
				& \text{if }s\nmid(n-n_0),\\
				(1-t)h(\Delta_{n-1};t)+t h(\Delta_{n-s};t)
				& \text{if }s\mid(n-n_0).
			\end{cases}
		\end{equation}
		Comparing coefficients proves (iii).
		
		Finally, recall that
		\[
		\widetilde{\chi}(\Delta)
		=\sum_{i\ge-1}(-1)^i f_i(\Delta).
		\]
		Using (ii), we obtain
		\begin{align*}
			\widetilde{\chi}(\Delta_n)
			&=\sum_{i\ge-1}(-1)^i
			\bigl(f_i(\Delta_{n-1})+f_{i-1}(\Delta_{n-s})\bigr)
			=\widetilde{\chi}(\Delta_{n-1})
			-\widetilde{\chi}(\Delta_{n-s}).
		\end{align*}
		Since $\chi(\Delta)=\widetilde{\chi}(\Delta)+1$, the corresponding recurrence for the ordinary Euler characteristic follows. This proves (iv).
	\end{proof}
	
	Specializing \Cref{thm:clique} to chains of sparsity index $1$ gives the following consequence.
	
	\begin{cor}
		\label{cor:clique-f-h-sp=1}
		Let $\Gc=(G_n)_{n\ge1}$ be an $\Inc$-invariant chain of graphs with $\spi(\Gc)=1$. Then, for all $n\gg0$, the following statements hold:
		\begin{enumerate}
			\item $\omega(G_n)=\omega(G_{n-1})+1$.
			\item The last entry of the $f$-vector of $\Cl(G_n)$ is constant.
			\item The $h$-polynomial of $\Cl(G_n)$ is independent of $n$.
			\item $\chi(\Cl(G_n))=1$ and $\widetilde{\chi}(\Cl(G_n))=0$.
		\end{enumerate}
	\end{cor}
	
	\begin{proof}
		Part (i) follows immediately from \Cref{thm:clique}(i).
		
		To prove (ii), set $\Delta_n=\Cl(G_n)$ and $d_n=\omega(G_n)$. The last entry of the $f$-vector of $\Delta_n$ is therefore $f_{d_n-1}(\Delta_n)$. Since $d_n=d_{n-1}+1$, we have $f_{d_n-1}(\Delta_{n-1})=f_{d_{n-1}}(\Delta_{n-1})=0$. By \Cref{thm:clique}(ii),
		\begin{align*}
			f_{d_n-1}(\Delta_n)
			&=f_{d_n-1}(\Delta_{n-1})+f_{d_n-2}(\Delta_{n-1})
			=f_{d_n-2}(\Delta_{n-1})
			=f_{d_{n-1}-1}(\Delta_{n-1}).
		\end{align*}
		Thus, the last entry of the $f$-vector of $\Delta_n$ is eventually constant, proving (ii).
		
		Since $s=\spi(\Gc)=1$, \eqref{eq:clique-h} becomes
		\[
		h(\Delta_n;t)
		=(1-t)h(\Delta_{n-1};t)+t h(\Delta_{n-1};t)
		=h(\Delta_{n-1};t),
		\]
		which proves (iii).
		
		Finally, let $d$ denote the eventual degree of $h(\Delta_n;t)$. Since $d_n$ is increasing for $n\gg0$, we may choose $k$ large enough such that $d_k>d$. Then $d_n\ge d_k>d$ for all $n\ge k$. It follows that
		\[
		\widetilde{\chi}(\Delta_n)=(-1)^{d_n-1}h_{d_n}(\Delta_n)=0
		\quad\text{for all } n\ge k.
		\]
		Consequently, ${\chi}(\Delta_n)=\widetilde{\chi}(\Delta_n)+1=1$ for all $n\ge k$. The proof is complete.
	\end{proof}
	
	The next result gives further properties of chains with sparsity index $1$.
	
	\begin{prop}
		\label{prop:clique-sp=1}
		Let $\Gc=(G_n)_{n\ge1}$ be an $\Inc$-invariant chain of graphs with $\ind(\Gc)=r$ and $\spi(\Gc)=1$. Then the following statements hold for all $n\ge 5r$:
		\begin{enumerate}
			\item The number of facets of $\Cl(G_n)$ is constant.
			\item There is a bijection between the minimal nonfaces of $\Cl(G_n)$ and those of $\Cl(G_{n+1})$.
		\end{enumerate}
	\end{prop}
	
	The proof uses the following lemma, which was essentially proved in \cite{HHLNN}.
	
	\begin{lem}
		\label{lem:sp(I)=1}
		Keep the assumptions of \Cref{prop:clique-sp=1}. For $n\ge4r$, set
		$U_n = [2r, n-2r].$
		Then the following statements hold:
		\begin{enumerate}
			\item Every vertex of $U_n$ is isolated in $G_n^c$ for all $n\ge4r$.
			\item $G_n\setminus U_n\cong G_{n+1}\setminus U_{n+1}$ for all $n\ge5r$.
			\item $G_n^c\setminus U_n\cong G_{n+1}^c\setminus U_{n+1}$ for all $n\ge5r$.
		\end{enumerate}
	\end{lem}
	
	\begin{proof}
		Part (i) is \cite[Lemma~7.4]{HHLNN}. Part (ii) follows from \Cref{lem:del-graph-iso}, while part (iii) follows from part (ii) by taking complements. A slightly weaker version of part (iii) was also proved in \cite[Lemma~7.5]{HHLNN}.
	\end{proof}
	
	We can now prove \Cref{prop:clique-sp=1}.
	
	\begin{proof}[Proof of \Cref{prop:clique-sp=1}]
		Fix $n\ge5r$ and set $U_n=[2r,n-2r]$.
		
		(i) By \Cref{lem:sp(I)=1}(i), each vertex of $U_n$ is adjacent to every other vertex of $G_n$. Hence, every maximal clique of $G_n$ contains $U_n$. Moreover, a set $W\subseteq[n]\setminus U_n$ is a maximal clique of $G_n\setminus U_n$ if and only if $W\cup U_n$ is a maximal clique of $G_n$. Therefore,
		\[
		\nu(\Cl(G_n))=\nu(\Cl(G_n\setminus U_n)),
		\]
		where $\nu(\cdot)$ denotes the number of facets. By \Cref{lem:sp(I)=1}(ii),
		\[
		\nu(\Cl(G_{n+1}\setminus U_{n+1}))
		=\nu(\Cl(G_n\setminus U_n)),
		\]
		and the assertion follows.
		
		(ii) Since $\Cl(G_n)=\IN(G_n^c)$, the minimal nonfaces of $\Cl(G_n)$ are precisely the edges of $G_n^c$. By \Cref{lem:sp(I)=1}(i), $U_n$ consists of isolated vertices of $G_n^c$, and hence
		$E(G_n^c)=E(G_n^c\setminus U_n).$
		The desired bijection now follows from \Cref{lem:sp(I)=1}(iii).
	\end{proof}
	
	\begin{rem}
		Keep the assumptions of \Cref{prop:clique-sp=1}.
		\begin{enumerate}
			\item \Cref{prop:clique-sp=1}(ii) is a refinement of \Cref{cor:clique-f-h-sp=1}(iii). Indeed, by \Cref{lem:del-graph-iso}, the bijection in \Cref{prop:clique-sp=1}(ii) is induced by $\sigma_k$ for any $k\in [2r, n-2r].$ Consequently, up to relabeling, the Stanley--Reisner ring of $\Cl(G_{n+1})$ is obtained from that of $\Cl(G_{n})$ by adjoining one variable. Hence, the $h$-polynomials of $\Cl(G_n)$ and $\Cl(G_{n+1})$ coincide.
			
			\item In general, $\Cl(G_n)$ need not be pure. Therefore, \Cref{cor:clique-f-h-sp=1}(ii) and \Cref{prop:clique-sp=1}(i) are distinct statements. For example, let $\mathcal{G}=(G_n)_{n\ge1}$ be the $\Inc$-invariant chain with $\ind(\mathcal{G})=5$ and
			\[
			E(G_5)=\{(1,3),(2,4),(4,5)\}.
			\]
			Then $\spi(\Gc)=1$. Computations with Macaulay2 \cite{GS} yield the following table.
			
			\begin{table}[htbp]
				\centering
				\begin{tabular}{|c|c|c|}
					\hline
					$n$ & $f$-vector of $\Cl(G_n)$ & $\nu(\Cl(G_n))$ \\
					\hline
					5  & $(1,5,3)$              & 3 \\
					\hline
					6  & $(1,6,8,2)$            & 5 \\
					\hline
					7  & $(1,7,14,10,2)$        & 5 \\
					\hline
					8  & $(1,8,21,24,12,2)$     & 5 \\
					\hline
					9  & $(1,9,29,45,36,14,2)$  & 5 \\
					\hline
					10 & $(1,10,38,74,81,50,16,2)$ & 5 \\
					\hline
				\end{tabular}
				\caption{A chain with sparsity index $1$}
				\label{tab:clique-sp=1}
			\end{table}
			
			Thus, for $6\le n\le10$, the last entry of the $f$-vector of $\Cl(G_n)$ is $2$, while $\nu(\Cl(G_n))=5$.
		\end{enumerate}
	\end{rem}
	
	Motivated by \Cref{conj:IN-facet} and \Cref{prop:clique-sp=1}, we conclude this section with the following problem.
	
	\begin{prob}
		Let $\Gc=(G_n)_{n\ge1}$ be an $\Inc$-invariant chain of graphs. Determine the asymptotic behavior of the number of facets of $\Cl(G_n)$.
	\end{prob}

	%%%%%%%%%%%%%%%%%%%%%%%%%%%%%%%%%%%%%%%%%%%%%%%%%%%%%%%%
	
	\section{Chromatic numbers}\label{subsec.chro}
	
	In this section, we study the asymptotic behavior of the chromatic numbers of graphs in an $\Inc$-invariant chain. Our first result shows that, like the clique number, the chromatic number is eventually quasi-linear.
	
	\begin{thm}
		\label{thm:chromatic}
		Suppose $\Gc=(G_n)_{n\ge1}$ is an $\Inc$-invariant chain of graphs with $\ind(\Gc)=r$ and $\spi(\Gc)=s$. Then $\gamma(G_n)$ is eventually quasi-linear with period $s$. More precisely, there exists $n_0\in\N$ such that
		\[
		\gamma(G_n)
		=\gamma(G_{n_0})
		+\left\lfloor\frac{n-n_0}{s}\right\rfloor
		\quad\text{for all }n\ge n_0.
		\]
	\end{thm}
	
	This theorem is illustrated by \Cref{tab:clique} in \Cref{ex:clique}. Its proof relies on \Cref{lem:quasi-linear}. We first establish the inequalities needed to apply that result.

	\begin{lem}
		\label{lem:chromatic-upper}
		Let $\Gc=(G_n)_{n\ge1}$ be an $\Inc$-invariant chain of graphs with $\ind(\Gc)=r$ and $\spi(\Gc)=s$. Then
		\[
		\gamma(G_{n-1})\le\gamma(G_n)\le\gamma(G_{n-s})+1
		\quad \text{for all } n\ge 5r+3s.
		\]
	\end{lem}
	
	\begin{proof}
		Since $G_{n-1}$ is a subgraph of $G_n$, we have
		$\gamma(G_{n-1})\le\gamma(G_n).$
		It remains to prove the upper bound. Set $m=n-s$ and $c=\gamma(G_m)$. Choose independent sets $A_1,\dots,A_c$ of $G_m$ such that
		\[
		[m]=A_1\cup\cdots\cup A_c.
		\]
		Enlarging each $A_i$ if necessary, we may assume that every $A_i$ is a facet of $\IN(G_m)$.
		Write
		\[
		\alpha=2r+s-1
		\quad\text{and}\quad
		\beta=m-2r-s+1=n-2r-2s+1.
		\]
		Since $s\le r-1$ and $n\ge5r+3s$, we have $m\ge4r+3s$. Therefore, \Cref{lem:IN-facet-dec} shows that, for each $i\in[c]$, either
		\[
		A_i=I_t
		\quad\text{for some }t\in[2r,\beta],
		\]
		or
		\[
		A_i\in\FF\bigl(\IN(G_m\setminus[\alpha,\beta])\bigr).
		\]
		Moreover,
		\[
		\beta-\alpha=n-4r-3s+2\ge r.
		\]
		Fix $k\in [\alpha,\beta]$. Then	for each $j\in\{0,\dots,s-1\}$, \Cref{lem:del-graph-iso} gives an isomorphism
		\[
		\sigma_k\colon
		G_{m+j}\setminus[\alpha,\beta+j]
		\xrightarrow{\ \cong\ }
		G_{m+j+1}\setminus[\alpha,\beta+j+1].
		\]
		Iterating these isomorphisms yields
		\[
		\varphi:=\sigma_k^s\colon
		G_m\setminus[\alpha,\beta]
		\xrightarrow{\ \cong\ }
		G_n\setminus[\alpha,\beta+s],
		\]
		where
		\[
		\varphi(v)=
		\begin{cases}
			v   & \text{if }v<\alpha,\\
			v+s & \text{if }v>\beta.
		\end{cases}
		\]
		For each $i\in[c]$, define
		\[
		B_i=
		\begin{cases}
			A_i         & \text{if }A_i=I_t\text{ for some }t\in[2r,\beta],\\
			\varphi(A_i) & \text{if }A_i\in\FF\bigl(\IN(G_m\setminus[\alpha,\beta])\bigr).
		\end{cases}
		\]
		In the first case, $B_i=I_t$ is a facet of $\IN(G_n)$ by \Cref{lem:IN-facet-dec}. In the second case, $\varphi$ maps $A_i$ to a facet of
		$\IN(G_n\setminus[\alpha,\beta+s]),$
		which is again a facet of $\IN(G_n)$ by \Cref{lem:IN-facet-dec}. Thus, every $B_i$ is independent in $G_n$.
		
		It remains to determine which vertices are covered by the sets $B_i$. After relabeling, we may assume that
		\[
		A_i=I_t\text{ for some }t\in[2r,\beta]
		\quad\Longleftrightarrow\quad i\in[d]
		\]
		for some $d\in[c]$. Set
		\[
		T=\bigcup_{i\in[d]}A_i.
		\]
		Since $A_1,\dots,A_c$ cover $[m]$ and $A_j\cap[\alpha,\beta]=\emptyset$ for every $j\in[d+1,c]$, we obtain
		$$[\alpha,\beta]\subseteq T\subseteq\bigcup_{i\in[c]}B_i.$$
		The map $\varphi$ fixes every vertex of $[\alpha-1]$, so
		\[
		[\beta]=[\alpha-1]\cup [\alpha,\beta]\subseteq \bigcup_{i\in[c]}B_i.
		\]
		Because each interval occurring in $T$ has the form $I_t$ with $t\le\beta$, there exists $q\in\{0,\dots,s-1\}$ such that
		\[
		T\cap[\beta+1,m]=[\beta+1,\beta+q],
		\]
		where the interval on the right is understood to be empty when $q=0$. Since the sets $A_i$ cover $[m]$, it follows that 
		\[
		[\beta+q+1,m]\subseteq\bigcup_{j=d+1}^c A_j.
		\]
		Applying $\varphi$, we obtain
		\[		
		[\beta+q+s+1,n]=\varphi\bigl([\beta+q+1,m]\bigr)
		\subseteq\varphi\left(\bigcup_{j=d+1}^c A_j\right)
		\subseteq \bigcup_{i\in[c]}B_i.
		\]
		Consequently,
		\[
		[n]\setminus\bigcup_{i\in[c]}B_i
		\subseteq[\beta+q+1,\beta+q+s]
		=I_{\beta+q+1}.
		\]
		By \Cref{lem:independent-interval}, the interval
		$I_{\beta+q+1}$ is independent in $G_n$. Thus, $[n]$ is covered by
		$c+1$ independent sets. Therefore,
		\[
		\gamma(G_n)\le c+1=\gamma(G_{n-s})+1,
		\]
		which completes the proof.
	\end{proof}

	Let us now prove \Cref{thm:chromatic}.
	
	\begin{proof}[Proof of \Cref{thm:chromatic}]
		For every $n\ge1$, we have
		$\gamma(G_n)\ge \omega(G_n).$
		By \Cref{thm:clique}, $\omega(G_n)$ is eventually quasi-linear with slope $1/s$. On the other hand, \Cref{lem:chromatic-upper} gives
		\[
		\gamma(G_{n-1})\le\gamma(G_n)\le\gamma(G_{n-s})+1
		\]
		for all sufficiently large $n$. The conclusion now follows from \Cref{lem:quasi-linear}(iv).
	\end{proof}
	
	For every integer $k\ge2$, Mycielski \cite{My55} constructed a graph $G$ with $\omega(G)=2$ and $\gamma(G)=k$; see also \cite[Theorem~14.11]{BM}. Thus, both the difference $\gamma(G)-\omega(G)$ and the ratio $\gamma(G)/\omega(G)$ can be arbitrarily large. For graphs in an $\Inc$-invariant chain, however, these quantities are eventually small.
	
	\begin{thm}
		\label{prop:gamma-omega}
		Let $\Gc=(G_n)_{n\ge1}$ be an $\Inc$-invariant chain of graphs with $\ind(\Gc)=r$. Then
		\[
		\gamma(G_n)\le\omega(G_n)+1
		\quad\text{for all }n\ge r(4r-3).
		\]
	\end{thm}
	
	The main step in the proof of this result is to show that deleting a suitable independent set from $G_n$ yields a perfect graph. In fact, we prove the following stronger statement.
	
	\begin{prop}
		\label{lem:facet-del}
		Let $\Gc=(G_n)_{n\ge1}$ be an $\Inc$-invariant chain of graphs with $\ind(\Gc)=r$ and $\spi(\Gc)=s$. Let $I_t=[t,t+s-1]$ be an interval contained in $[2r,n-2r]$. Then $G_n\setminus I_t$ is weakly chordal for every $n\ge r(4r-3)$.
	\end{prop}
	
	The proof of this proposition relies on the following auxiliary result from \cite[Lemma~3.9]{HHLNN} and \cite[Proposition~4.1]{HNT2024}.
	
	\begin{lem}
		\label{lem:no-induced-cycle}
		Let $\Gc=(G_n)_{n\ge1}$ be an $\Inc$-invariant chain of graphs with $\ind(\Gc)=r$. Then the following statements hold:
		\begin{enumerate}
			\item $G_n$ contains no induced cycle $C_k$ with $k\ge6$
			whenever $n\ge3r$.
			\item $G_n^c$ contains no induced cycle $C_k$ with $k\ge5$ whenever $n\ge kr$.
		\end{enumerate}
	\end{lem}
	
	\setcounter{thm}{4} 
	
	\begin{proof}[Proof of \Cref{lem:facet-del}]
		Let $n\ge r(4r-3)$ and set
		$H_n=G_n\setminus I_t.$
		By \Cref{lem:no-induced-cycle}(i), the graph $G_n$, and hence its induced subgraph $H_n$, contains no induced cycle of length at least $6$.
		Since $C_5^c\cong C_5$, it remains to show that $H_n^c=G_n^c\setminus I_t$ contains no induced $C_k$ for every $k\ge 5$. 
		Set
		\[
		L=[r-1],
		\quad R=[n-r+1,n],
		\quad M=[r,n-r],
		\]
		and
		\[
		M_L=[r,2r-2],
		\qquad
		M_R=[n-2r+2,n-r].
		\]
		
		\begin{claim}
			\label{cl:boundary-neighbor}
			The following statements hold:
			\begin{enumerate}
				\item If $u\in L$, then $N_{G_n^c}(u)\cap M\subseteq M_L$.
				\item If $u\in R$, then $N_{G_n^c}(u)\cap M\subseteq M_R$.
			\end{enumerate}
		\end{claim}
		
		\begin{proof}[Proof of \Cref{cl:boundary-neighbor}]
			(i) It suffices to show that $(u,v)\in E(G_n)$ for every
			\[
			v\in M\setminus M_L=[2r-1,n-r].
			\]
			By \Cref{assm:main}, there exists $i\in[r]$ such that $(1,i)\in E(G_r)$. Since $v-u\ge r>i-1$, we have
			\[
			0\le u-1\le v-i\le n-r.
			\]
			Hence, $(u,v)\in E(G_n)$ by \Cref{lem:E(Gn)}.
			
			(ii) Similarly, it suffices to show that $(w,u)\in E(G_n)$ for every
			\[
			w\in M\setminus M_R=[r,n-2r+1].
			\]
			By \Cref{assm:main}, there exists $b\in[r]$ such that $(b,r)\in E(G_r)$. Since $u-w\ge r>r-b$, we have
			\[
			0\le w-b\le u-r\le n-r.
			\]
			Hence, $(w,u)\in E(G_n)$ by \Cref{lem:E(Gn)}.
		\end{proof}
		
		\begin{claim}
			\label{cl:monotone}
			Every induced path $P=(v_1,\dots,v_\ell)$ in $G_n^c[M]$ is monotone.
		\end{claim}
		
		\begin{proof}[Proof of \Cref{cl:monotone}]
			We may assume that $\ell\ge3$. Since $P$ is a path in $G_n^c[M]$, the pair $\{v_i,v_{i+1}\}$ is a nonedge of $G_n$ for every $i\in[\ell-1]$. By \Cref{lem:independent-interval}(i),
			$|v_{i+1}-v_i|<s.$
			Since $P$ is induced, $\{v_i,v_{i+2}\}\in E(G_n)$ for every $i\in[\ell-2]$. Thus, \Cref{lem:independent-interval}(ii) gives
			$|v_{i+2}-v_i|\ge s.$
			Consequently, the differences $v_{i+1}-v_i$ and $v_{i+2}-v_{i+1}$ have the same sign, since otherwise,
			\[
			|v_{i+2}-v_i|
			<\max\{|v_{i+2}-v_{i+1}|,|v_{i+1}-v_i|\}
			<s,
			\]
			a contradiction. Therefore, $P$ must be monotone. 
		\end{proof}
		
		\begin{claim}
			\label{cl:induced-cycle}
			If $H_n^c$ contains an induced cycle $C_k$ of length $k$, then $k\le4r-3$.
		\end{claim}
		
		\begin{proof}[Proof of \Cref{cl:induced-cycle}]
			Let $P$ be a connected component of the subgraph of $C_k$ induced by the vertex set $V(C_k)\cap M$. Then $P$ is an induced path in $G_n^c[M]$ and hence is monotone by \Cref{cl:monotone}. In particular, $P\ne C_k$. Let $x,y$ with $x<y$ be the two neighbors of $P$ on $C_k$. Then both $x$ and $y$ lie in $L\cup R$. We distinguish three cases.
			
			\smallskip
			\emph{Case 1}: $x,y\in L$. Then both endpoints of $P$ lie in $M_L$ by \Cref{cl:boundary-neighbor}. Since $P$ is monotone, it follows that $P\subseteq M_L$.
			
			\smallskip
			\emph{Case 2}: $x,y\in R$. Similarly to Case 1, we obtain $P\subseteq M_R$.
			
			\smallskip
			\emph{Case 3}: $x\in L$ and $y\in R$. Reversing $P$ if necessary, we may assume that $P=(v_1,\dots,v_\ell)$ with $v_1<\cdots<v_\ell$. By \Cref{cl:boundary-neighbor}, 
			\[
			v_1\le 2r-2\quad\text{and}\quad v_\ell\ge n-2r+2.	
			\]
			This implies $I_t\subseteq [2r,n-2r]\subseteq [v_1,v_\ell].$ As shown in the proof of \Cref{cl:monotone}, the distance between any two consecutive vertices of $P$ is less than $s$. Since $I_t$ has length $s$, there must be a vertex of $P$ lying in $I_t$. This contradicts the assumption that $C_k$ is a cycle in $H_n^c=G_n^c\setminus I_t$. Therefore, this case cannot occur.
			
			Applying the preceding argument to every connected component of $C_k$ induced by $V(C_k)\cap M$, we conclude that
			\[
			V(C_k)\subseteq L\cup M_L\cup R\cup M_R=[2r-2]\cup[n-2r+2,\,n].
			\]
			Consequently, $k\le (2r-2)+(2r-1)=4r-3$, as claimed.
		\end{proof}
		
		Let us now resume the proof of the proposition. Since $n\ge r(4r-3)$, it follows from \Cref{lem:no-induced-cycle} that  $G_n^c$, and hence $H_n^c$, contains no induced $C_k$ with $5\le k\le 4r-3$. Combined with \Cref{cl:induced-cycle}, we deduce that $H_n^c$ has no induced $C_k$ for every $k\ge5$. This completes the proof.	
	\end{proof}
	
	We now present the proof of \Cref{prop:gamma-omega}.
	
	\begin{proof}[Proof of \Cref{prop:gamma-omega}]
		Let $n\ge r(4r-3)$ and choose an interval $I_t\subseteq [2r,n-2r]$. By \Cref{lem:facet-del}, the graph $G_n\setminus I_{t}$ is weakly chordal and therefore perfect. Consequently,
		\[
		\gamma(G_n\setminus I_t)=\omega(G_{n}\setminus I_t).
		\]
		Since $I_t$ is independent in $G_n$, we obtain
		\[
		\gamma(G_n)\le \gamma(G_n\setminus I_t)+1=\omega(G_{n}\setminus I_t)+1\le\omega(G_{n})+1.	
		\]	
		The proof is complete. 
	\end{proof}
	
	\setcounter{thm}{5} 
	
	Given a graph $G$, let $P_G$ denote its \emph{chromatic polynomial}. Thus, for each $k\in\N$, the value $P_G(k)$ is the number of proper colorings of $G$ using $k$ colors. In particular,
	\[
	\gamma(G)=\min\{k\in\N\mid P_G(k)>0\}.
	\]
	We conclude this section with the following problem.
	
	\begin{prob}
		Let $\Gc=(G_n)_{n\ge1}$ be an $\Inc$-invariant chain of graphs. Study the asymptotic behavior of the chromatic polynomials $P_{G_n}$.
	\end{prob}

	%-----------------------------------------------
	\section{Matching numbers} \label{subsec.mat} 
	
	Recall that the matching number of a graph on $n$ vertices is at most $\lfloor n/2\rfloor$. The main result of this section shows that every $\Inc$-invariant chain of graphs eventually attains this upper bound.
	
	\begin{thm}
		\label{thm:matching-number}
		Let $\Gc=(G_n)_{n\ge1}$ be an $\Inc$-invariant chain of graphs with $\ind(\Gc)=r$. Then
		\[
		\mu(G_n)=\left\lfloor\frac{n}{2}\right\rfloor
		\quad\text{for all }n\ge8r.
		\]
	\end{thm}
	
	\begin{proof}   
		We first establish the following augmentation property.
		
		\begin{claim}
			\label{cl:increasing-matching-number}
			If $n\ge4r$ and $2\mu(G_n)<n$, then $\mu(G_{n+1})>\mu(G_n)$.
		\end{claim}  
		
		\begin{proof}[Proof of \Cref{cl:increasing-matching-number}]
			Write $b=b(\Gc)$ and $B=B(\Gc)$, and set
			\[
			B'=n-r+B+1.
			\]
			By \Cref{lem.GWeaklyChordal}(i), the neighborhood of $n+1$ in $G_{n+1}$ is
			\[
			N\defas N_{G_{n+1}}(n+1)=[b,B'].
			\]
			Let $M$ be a maximum matching of $G_n$, and let $U$ be the set of vertices covered by $M$. If some vertex $x\in N$ is not covered by $M$, then
			\[
			M\cup\{(x,n+1)\}
			\]
			is a matching of $G_{n+1}$, and the claim follows immediately. We may therefore assume that $N\subseteq U$.
			
			Set $V=[n]\setminus N$. Since
			\[
			|N|=B'-b+1=n-r+B-b+2>n-r\ge3r,
			\]
			we have
			\[
			|V|=n-|N|=r-B+b-2<r.
			\]
			Partition $N$ into the sets
			\begin{align*}
				V_1
				&=\bigl\{v\in N\mid \{v,w\}\in M
				\text{ for some }w\in V\bigr\},\\
				V_2
				&=\bigl\{v\in N\mid \{v,w\}\in M
				\text{ for some }w\in N\bigr\}.
			\end{align*}
			Because every vertex of $N$ is covered by $M$, we have $N=V_1\sqcup V_2$. Moreover, since $M$ is a matching,
			\[
			|V_1|\le|V|<r
			\quad\text{and}\quad
			|V_2|=|N|-|V_1|>2r.
			\]
			Consider the interval
			\[
			V_3=[b+r,B'-r]\subseteq N.
			\]
			Since $|N\setminus V_3|=2r$ and $|V_2|>2r$, there exists a vertex $v\in V_2\cap V_3$. By definition of $V_2$, we have $\{v,w\}\in M$ for some $w\in N$. 
			The hypothesis of the claim gives $|U|=2\mu(G_n)<n$. Hence, there exists an uncovered vertex $u\in [n]\setminus U$. Since $N\subseteq U$, we have $u\in [n]\setminus N =V$. 
			Thus, either $u\le b-1$ or $u\ge B'+1$. We consider these two cases.
			
			\smallskip
			\emph{Case 1}: $u\le b-1$. Since $v \in V_3$, we have $v \ge b+r > u+r$. By \Cref{assm:main}, there exists $i\in[r]$ with $(1,i)\in E(G_r)$. Thus, $(1,n-r+i)\in E(G_n)$ by \Cref{lem:E(Gn)}. Since
			\[
			1\le u<v\le B'-r=n-2r+B+1\le n-r+i,
			\]
			it follows from \Cref{lem:gap} that $(u,v)\in E(G_n)$. 
			
			\smallskip
			\emph{Case 2}: $u\ge B'+1$. As $v \in V_3$, we have $v\le B'-r$, hence $u>v+r$. Since $(b,n)\in E(G_n)$ and
			\[
			b<v<u\le n,
			\]
			it follows again from \Cref{lem:gap} that $(v,u)\in E(G_n)$.
			
			In either case, $\{u,v\}\in E(G_n)$. Since $w\in N$, we also have $\{w,n+1\}\in E(G_{n+1})$. Consequently,
			\[
			M'
			=\bigl(M\setminus\{\{v,w\}\}\bigr)
			\cup\bigl\{\{u,v\},\{w,n+1\}\bigr\}
			\]
			is a matching of $G_{n+1}$ with $|M'|=|M|+1$. Therefore,
			$\mu(G_{n+1})>\mu(G_n),$
			proving the claim.
		\end{proof}
		
		We now complete the proof of the theorem. For $n\ge4r$, set
		\[
		a_n
		=\min\left\{n-4r,
		\left\lfloor\frac{n}{2}\right\rfloor\right\}.
		\]
		We prove by induction on $n$ that
		\[
		\mu(G_n)\ge a_n
		\quad\text{for all }n\ge4r.
		\]
		The assertion is clear for $n=4r$. Suppose that it holds for some $n\ge4r$. If $2\mu(G_n)<n$, then the claim gives
		\[
		\mu(G_{n+1})\ge\mu(G_n)+1\ge a_n+1\ge a_{n+1},
		\]
		where the last inequality follows directly from the definition of $a_n$. Otherwise, since $2\mu(G_n)\le n$, we have $2\mu(G_n)=n$. Thus, $n$ is even and
		\[
		\mu(G_n)=\frac{n}{2}
		=\left\lfloor\frac{n+1}{2}\right\rfloor.
		\]
		Since $G_n$ is a subgraph of $G_{n+1}$, we obtain
		\[
		\mu(G_{n+1})\ge\mu(G_n)
		=\left\lfloor\frac{n+1}{2}\right\rfloor
		\ge a_{n+1}.
		\]
		This proves the induction claim.
		
		Finally, if $n\ge8r$, then $n-4r\ge\lfloor n/2\rfloor$.
		It follows that
		\[
		\mu(G_n)\ge a_n=\left\lfloor\frac{n}{2}\right\rfloor.
		\]
		Since the reverse inequality holds for every graph on $n$ vertices, we conclude that
		\[
		\mu(G_n)=\left\lfloor\frac{n}{2}\right\rfloor
		\quad\text{for all }n\ge8r.
		\qedhere
		\]
	\end{proof}
	
	\begin{rem}
		An \emph{induced matching} of a graph $G$ is a matching $M$ such that the subgraph induced by the vertices covered by $M$ has edge set precisely $M$. The \emph{induced matching number} of $G$ is the maximum cardinality of an induced matching. While \Cref{thm:matching-number} shows that the matching number along an $\Inc$-invariant chain  $\Gc=(G_n)_{n\ge1}$ can grow arbitrarily large, \cite[Theorem~3.1]{HNT2024} establishes a striking contrast: the induced matching number of $G_n$ is constant and belongs to $\{1,2\}$ for all $n\ge3r$.
	\end{rem}
	
	The \emph{matching complex} of a graph $G$, denoted by $M(G)$, is the simplicial complex with vertex set $E(G)$ whose faces are the matchings of $G$. Its dimension is
	\[
	\dim M(G)=\mu(G)-1.
	\]
	Thus, \Cref{thm:matching-number} determines the eventual dimension of $M(G_n)$:
	\[
	\dim M(G_n)=\left\lfloor\frac{n}{2}\right\rfloor-1
	\quad\text{for all }n\ge8r.
	\]
	In light of the results in \Cref{subsec.indep,subsec.cliq}, it is natural to investigate finer asymptotic properties of these complexes.
	
	\begin{prob}
		Let $\Gc=(G_n)_{n\ge1}$ be an $\Inc$-invariant chain of graphs. Determine the asymptotic behavior of the matching complexes $M(G_n)$.
	\end{prob}

	%-----------------------------------------------
	\section{Minimal and admissible paths}\label{subsec.path}
	
	In this section, we investigate the asymptotic behavior of the maximal lengths of admissible and minimal paths in an $\Inc$-invariant chain of graphs. These paths play a central role in the Gr\"obner basis theory of binomial edge ideals \cite{CDG,HHHKR,O11}.
	
	\begin{defn}
		Let $G$ be a simple graph on $[n]$, and let $u<v$ be vertices of $G$. A path
		\[
		P=(u_0,u_1,\dots,u_q)
		\]
		from $u_0=u$ to $u_q=v$ is called \emph{admissible} if the following conditions hold:
		\begin{enumerate}
			\item $u_k\ne u_l$ whenever $k\ne l$;
			\item for every proper subset
			\[
			\{v_1,\dots,v_t\}\subsetneq\{u_1,\dots,u_{q-1}\},
			\]
			the sequence $u,v_1,\dots,v_t,v$ is not a path in $G$;
			\item every internal vertex $u_k$, $1\le k\le q-1$, satisfies either $u_k<u$ or $u_k>v$.
		\end{enumerate}
		A path satisfying only conditions \textup{(i)} and \textup{(ii)} is called \emph{minimal}.
	\end{defn}
	
	Minimal paths are called \emph{weakly admissible} in \cite{BBS}. We illustrate these notions with the following example.
	
	\begin{ex}
		Let $G$ be the graph depicted in \Cref{fig:cycle}. Then:
		\begin{enumerate}
			\item The path $P_1=(1,4,3,5)$ is not minimal because of the chord $(4,5)$.
			\item The path $P_2=(1,2,5,3)$ is minimal but not admissible, since the internal vertex $2$ lies between the endpoints $1$ and $3$.
			\item The path $P_3=(2,1,4,3)$ is admissible.
		\end{enumerate}
	\end{ex}
	
	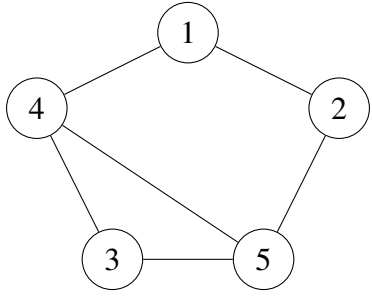
\begin{figure}[H]
		\centering
		\begin{tikzpicture}[
			vertex/.style={circle,draw,fill=white,minimum size=8mm,inner sep=0pt},
			every edge/.style={draw,thick}
			]
			\node[vertex] (1) at (2,3) {$1$};
			\node[vertex] (2) at (4,2) {$2$};
			\node[vertex] (5) at (3,0) {$5$};
			\node[vertex] (4) at (1,0) {$3$};
			\node[vertex] (3) at (0,2) {$4$};
			
			\draw (1)--(2);
			\draw (2)--(5);
			\draw (5)--(4);
			\draw (4)--(3);
			\draw (3)--(1);
			\draw (3)--(5);
		\end{tikzpicture}
		\caption{A graph on five vertices}
		\label{fig:cycle}
	\end{figure}
	
	The main result of this section provides optimal uniform bounds on the lengths of admissible and minimal paths in an $\Inc$-invariant chain. It also shows that their maximal lengths eventually stabilize.
	
	\begin{thm}
		\label{prop:ad-path-length}
		Let $\Gc=(G_n)_{n\ge1}$ be an $\Inc$-invariant chain of graphs with $\ind(\Gc)=r$. Let $\ell(G_n)$ and $\ell'(G_n)$ denote the maximal lengths of an admissible path and a minimal path in $G_n$, respectively. Then the following statements hold:
		\begin{enumerate}
			\item For every $n\ge5r$,
			\[
			\ell(G_n)\le3
			\quad\text{and}\quad
			\ell'(G_n)\le5.
			\]
			\item For every $n\ge5r+9$,
			\[
			\ell(G_n)=\ell(G_{n+1})
			\quad\text{and}\quad
			\ell'(G_n)=\ell'(G_{n+1}).
			\]
		\end{enumerate}
	\end{thm}
	
	The next example shows that the bounds in \Cref{prop:ad-path-length}(i) are sharp.
	
	\begin{ex}
		Let $\Gc=(G_n)_{n\ge1}$ be an $\Inc$-invariant chain of graphs.
		\begin{enumerate}
			\item Suppose that $\ind(\Gc)=4$ and
			\[
			E(G_4)=\{(1,3),(1,4),(2,4)\}.
			\]
			Using \Cref{lem:E(Gn)}, one checks that
			\[
			P=(2,4,1,3)
			\]
			is an admissible path in $G_n$ for every $n\ge4$.
			
			\item Suppose that $\ind(\Gc)=8$ and
			\[
			E(G_8)=\{(1,3),(1,4),(4,5),(5,8),(6,8)\}.
			\]
			Again using \Cref{lem:E(Gn)}, one checks that
			\[
			P=(3,1,4,n-3,n,n-2)
			\]
			is a minimal path in $G_n$ for every $n\ge8$.
		\end{enumerate}
	\end{ex}
	
	The proof of \Cref{prop:ad-path-length} requires several auxiliary results. We begin with a simple property of admissible paths.
	
	\begin{lem}
		\label{lem:single-edge}
		Let $\Gc=(G_n)_{n\ge1}$ be an $\Inc$-invariant chain of graphs with $\ind(\Gc)=r$. Suppose that $n\ge r$, and let $u,v\in[n]$ satisfy $v\ge u+r$. Then every admissible path from $u$ to $v$, if one exists, consists of the single edge $(u,v)$.
	\end{lem}
	
	\begin{proof}
		Let $P=(u_0,u_1,\dots,u_q)$ be an admissible path from $u=u_0$ to $v=u_q$. Set
		\[
		t=\max\{0\le k\le q-1\mid u_k\le u\}.
		\]
		Since no internal vertex of $P$ lies in $[u,v]$, we have
		$u_t\le u<v\le u_{t+1}.$
		As $(u_t,u_{t+1})\in E(G_n)$, \Cref{lem:gap} yields $(u,v)\in E(G_n)$. Condition \textup{(ii)} in the definition of admissibility now forces $P=(u,v)$.
	\end{proof}
	
	For the following auxiliary results, we retain the next assumptions.
	
	\begin{assm}
		\label{assm:path}
		Let $\Gc=(G_n)_{n\ge1}$ be an $\Inc$-invariant chain of graphs with $\ind(\Gc)=r$. Fix $n\ge5r$, and let
		$P=(u_0, u_1, \ldots, u_q)$
		be a minimal path in $G_n$ from $u=u_0$ to $v=u_q$, where $u<v$.
	\end{assm}
	
	To bound the length of $P$, we distinguish two cases according to whether $v\le n-r$ or $v>n-r$. We first record a useful consequence of minimality.
	
	\begin{lem}
		\label{lem:min-max}
		Under \Cref{assm:path}, suppose that
		\[
		\max\{u_{t-1},u_t\}<v\le n-r
		\]
		for some $t\in[q]$. Then $t=q-1$ and $u_{q-1}<u_{q-2}$.
	\end{lem}
	
	\begin{proof}
		Set
		\[
		x_t=\min\{u_{t-1},u_t\}
		\quad\text{and}\quad
		y_t=\max\{u_{t-1},u_t\}.
		\]
		Then $(x_t,y_t)\in E(G_n)$.
		By assumption, we have $x_t<y_t<v\le n-r$. A north move (\Cref{lem.shortmoves}(iii)) applied to $(x_t,y_t)$ gives $(x_t,v)\in E(G_n)$. By minimality, $x_t$ must be the immediate predecessor of $v$, so $x_t=u_{q-1}$. If $x_t=u_{t-1}$, then $t=q$, which would imply $y_t=\max\{u_{q-1},v\}\ge v$, a contradiction. Thus, $x_t=u_t$ and $t=q-1$. It follows that $u_{q-1}=\min\{u_{q-2},u_{q-1}\}<u_{q-2}.$
	\end{proof}

	\begin{lem}
		\label{lem:r<j<n-r}
		Under \Cref{assm:path}, if $v\le n-r$, then $q\le 3$.
	\end{lem}
	
	\begin{proof}
		Suppose that $q\ge4$. Since $P$ is minimal, no two nonconsecutive vertices of $P$ are adjacent. In particular,
		\begin{equation}
			\label{eq:non-adjacent}
			\{u,u_{q-2}\},\ \{u,u_{q-1}\},\ \{u,v\},\
			\{u_1,u_{q-1}\},\ \{u_1,v\}\notin E(G_n).
		\end{equation}
		Since $v\le n-r$ and $q-1>1$, \Cref{lem:min-max} implies that
		$\max\{u_0,u_1\}>v.$
		As $u_0=u<v$, it follows that $u<v<u_1$. Now $(u,u_1)\in E(G_n)$, whereas $(u,v),(v,u_1)\notin E(G_n)$. Hence, \Cref{lem:gap} gives
		\begin{equation}
			\label{eq:u-v-u1}
			v<u+r
			\quad\text{and}\quad
			u_1<v+r.
		\end{equation}
		We distinguish two cases.
		
		\smallskip
		\emph{Case 1: $u_{q-1}>v$.}
		If $u\ge r$, then $r\le u<v<u_{q-1}$, so a west move applied to $(v,u_{q-1})$ yields $(u,u_{q-1})\in E(G_n)$, contradicting \eqref{eq:non-adjacent}. Thus, $u<r$. Together with
		\eqref{eq:u-v-u1}, this gives
		\[
		v<u+r<2r
		\quad\text{and}\quad
		u_1<v+r<3r.
		\]
		If $u_{q-1}<u_1$, then
		\[
		v<u_{q-1}<u_1<3r<n-r.
		\]
		A north move applied to $(v,u_{q-1})$ gives $(v,u_1)\in E(G_n)$, contradicting \eqref{eq:non-adjacent}.
		Therefore, $u_1<u_{q-1}$.
		If $u_{q-1}\le n-r$, then a north move applied to $(u,u_1)$ gives
		$(u,u_{q-1})\in E(G_n)$, again contradicting
		\eqref{eq:non-adjacent}. Hence, $u_{q-1}>n-r$, and consequently	
		$u_{q-1}>4r>u_1+r.$
		Since $v<u_1<u_{q-1}$ and $(v,u_{q-1})\in E(G_n)$,
		\Cref{lem:gap} yields $(u_1,u_{q-1})\in E(G_n)$, another contradiction.
		
		\smallskip
		\emph{Case 2: $u_{q-1}<v$.}
		We first show that $u_{q-2}>u_{q-1}$. This is immediate if $u_{q-2}>v$. If $u_{q-2}<v$, then $\max\{u_{q-2},u_{q-1}\}<v\le n-r$, and \Cref{lem:min-max} gives $u_{q-2}>u_{q-1}$.
		Comparing $u_1$ and $u_{q-2}$, we therefore have either
		\[
		u<u_1<u_{q-2}
		\]
		or
		\[
		u_{q-1}<u_{q-2}<u_1.
		\]
		If $\max\{u_1,u_{q-2}\}\le n-r$, then a north move applied to $(u,u_1)$ in the first case, or to $(u_{q-1},u_{q-2})$ in the second, produces the chord $(u,u_{q-2})$ or $(u_{q-1},u_1)$, respectively. This contradicts \eqref{eq:non-adjacent}. Thus,
		\begin{equation}
			\label{eq:u1-uq-2}
			\max\{u_1,u_{q-2}\}>n-r\ge4r.
		\end{equation}
		By \eqref{eq:u-v-u1}, we have $v>u_1-r$. Moreover,
		$v>u_{q-2}-r.$
		Indeed, this is immediate if $u_{q-2}<v$. If $v<u_{q-2}$, it follows from
		\Cref{lem:gap}, since
		\[
		u_{q-1}<v<u_{q-2},\qquad
		(u_{q-1},u_{q-2})\in E(G_n),
		\quad\text{and}\quad
		(v,u_{q-2})\notin E(G_n).
		\]
		Combining these inequalities with \eqref{eq:u1-uq-2} and \eqref{eq:u-v-u1}, we obtain
		\[
		v>\max\{u_1,u_{q-2}\}-r>3r
		\quad\text{and}\quad
		u>v-r>2r.
		\]
		If $u<u_{q-1}$, then
		$r<u<u_{q-1}<v,$
		so a west move applied to $(u_{q-1},v)$ yields
		$(u,v)\in E(G_n)$, contradicting \eqref{eq:non-adjacent}.
		Therefore,
		\[
		u_{q-1}<u<v.
		\]
		Since $(u_{q-1},u)\notin E(G_n)$, \Cref{lem:gap} gives
		$u_{q-1}>u-r>r.$
		Thus,
		\[
		r<u_{q-1}<u<u_1.
		\]
		A west move applied to $(u,u_1)$ now yields
		$(u_{q-1},u_1)\in E(G_n)$, again contradicting
		\eqref{eq:non-adjacent}.
		
		Both cases lead to contradictions. Therefore, $q\le3$.
	\end{proof}

	\begin{lem}
		\label{cl:u<r}
		Under \Cref{assm:path}, suppose that $v>n-r$ and $q\ge4$. Then
		\[
		w\defas\max\{u_k\mid0\le k\le q-3\}<r.
		\]
	\end{lem}
	
	\begin{proof}
		Suppose that $w\ge r$. Let
		\[
		x=\min\{u_{q-1},v\}
		\quad\text{and}\quad
		y=\max\{u_{q-1},v\}.
		\]
		Then $(x,y)\in E(G_n)$, while path minimality implies that $w$ is adjacent to neither $x$ nor $y$. If $w<x$, a west move applied to $(x,y)$ gives $(w,y)\in E(G_n)$, a contradiction. Thus, $w>x$.
		This implies either $w > y$ or $x < w < y$; in the latter case, \Cref{lem:gap} requires $w > y-r$. Hence, we always have $w > y-r$.
		Consequently,
		\[
		w>y-r\ge v-r>n-2r.
		\]
		Since $q\ge4$, the vertex $w$ has a neighbor $u_l$ along the subpath $(u_0,\dots,u_{q-3})$. Minimality implies that $u_l$ is adjacent to neither $x$ nor $y$. We consider three possibilities for $u_l$.
		
		If $u_l>y$, then
		\[
		r<n-r<v\le y<u_l<w,
		\]
		so a west move applied to $(u_l,w)$ gives $(y,w)\in E(G_n)$, a contradiction.
		
		If $x<u_l<y$, then since $(x,u_l), (u_l,y) \notin E(G_n)$, \Cref{lem:gap} forces
		\[
		u_l>y-r>n-2r
		\quad\text{and}\quad
		x>u_l-r>n-3r>r.
		\]
		A west move applied to $(u_l,w)$ now gives $(x,w)\in E(G_n)$, a contradiction.
		
		Finally, if $u_l < x$, then $u_l < x < w$. Since $(u_l,x), (x,w) \notin E(G_n)$, \Cref{lem:gap} gives
		\[
		x>w-r>n-3r
		\quad\text{and}\quad
		u_l>x-r>n-4r\ge r.
		\]
		A west move applied to $(x,y)$ then yields $(u_l,y)\in E(G_n)$, again a contradiction. 
	\end{proof}

	\begin{lem}
		\label{lem:j>n-r}
		Under \Cref{assm:path}, if $v>n-r$, then $q\le 5$.
	\end{lem}
	
	\begin{proof}
		Suppose that $q\ge6$. By \Cref{cl:u<r}, all vertices of the subpath
		$P_1=(u_0,\dots,u_{q-3})$
		lie in $[r-1]$. After reversing $P_1$ if necessary, \Cref{lem:r<j<n-r} gives $q-3\le3$. Since $q\ge6$, it follows that $q=6$.
		
		If $u_4 \le n-r$, then the subpath $P_2 = (u_0, \dots, u_4)$ has length $4$ and satisfies the hypothesis of \Cref{lem:r<j<n-r} (after being reversed if necessary). This is impossible. Hence, $u_4>n-r$.
		
		Set 
		\[
		w=\max\{u_0,u_1,u_2\}.
		\]
		Then $w<r$ by \Cref{cl:u<r}, and minimality implies that $(w,u_4)\notin E(G_n)$. We claim that $w>u_3$. This is clear if $u_0>u_3$. Suppose that $u_0<u_3$. Since $u_3<r<n-r$, \Cref{lem:min-max} applied to the minimal path $(u_0,\dots,u_{3})$ implies that
		\[
		\max\{u_0,u_1\}>u_3.
		\]
		Hence, $w>u_3$, as claimed.
		
		Now since $u_4>n-r>2r>w+r$ and $(u_3,u_4)\in E(G_n)$, \Cref{lem:gap} gives $(w,u_4)\in E(G_n)$. This contradiction concludes the proof.
	\end{proof}

	\begin{lem}
		\label{lem:j>n-r-admissible}
		Under \Cref{assm:path}, if $v>n-r$ and $P$ is admissible, then $q\le 3$.
	\end{lem}
	
	\begin{proof}
		Suppose that $q\ge4$. Then $(u,v)\notin E(G_n)$. By \Cref{lem:single-edge},
		\[
		u>v-r>n-2r>r.
		\]
		If $u_{q-1}>v$, a west move applied to $(v,u_{q-1})$ gives $(u,u_{q-1})\in E(G_n)$, contradicting minimality. Thus, $u_{q-1}<v$, and admissibility gives $u_{q-1}<u$. Since $(u_{q-1},v)\in E(G_n)$ but $(u_{q-1},u)\notin E(G_n)$, \Cref{lem:gap} yields
		\[
		u_{q-1}>u-r>n-3r\ge2r.
		\]
		If $u_1>u_{q-1}$, set
		\[
		x=\min\{u,u_1\}
		\quad\text{and}\quad
		y=\max\{u,u_1\}.
		\]
		Then $r<u_{q-1}<x<y$, and a west move applied to $(x,y)$ gives $(u_{q-1},y)\in E(G_n)$, contradicting minimality. Hence,
		\[
		u_1<u_{q-1}<u.
		\]
		Since $(u_1,u_{q-1})\notin E(G_n)$, \Cref{lem:gap} gives
		\[
		u_1>u_{q-1}-r>r.
		\]
		Thus, $r<u_1<u_{q-1}<v$. A west move applied to $(u_{q-1},v)$ now gives $(u_1,v)\in E(G_n)$, again contradicting minimality.
	\end{proof}
	
	For the stabilization argument, we use the following elementary observation.
	
	\begin{lem}
		\label{lem:empty-path-interval}
		Let $a\in\Z_{\ge0}$ and $M,N,L\in\N$. Suppose that
		\[
		N-M+1\ge(a+1)L.
		\]
		Then every subset $X\subseteq[M,N]$ with $|X|\le a$ is disjoint from an interval of $L$ consecutive integers contained in $[M,N]$.
	\end{lem}
	
	\begin{proof}
		By assumption, the interval $[M,N]$ contains $a+1$ pairwise disjoint intervals of $L$ consecutive integers. Since $X$ has at most $a$ elements, at least one of these intervals is disjoint from $X$.
	\end{proof}
	
	We are now ready to verify \Cref{prop:ad-path-length}.
	
	\begin{proof}[Proof of \Cref{prop:ad-path-length}]
		The bound for minimal paths in part \textup{(i)} follows from
		\Cref{lem:r<j<n-r,lem:j>n-r}.
		The bound for admissible paths follows from \Cref{lem:r<j<n-r,lem:j>n-r-admissible}.
		
		We prove part \textup{(ii)}. Fix $n\ge5r+9$. We first show that
		\[
		\ell(G_n)\le\ell(G_{n+1})
		\quad\text{and}\quad
		\ell'(G_n)\le\ell'(G_{n+1}).
		\]
		Let
		$P=(u_0,\dots,u_q)$
		be a minimal path in $G_n$ with $u_0<u_q$. We claim that there exists an interval $[\alpha,\beta]\subseteq[n]$ such that
		\begin{equation}
			\label{eq:empty-source-interval}
			\beta-\alpha\ge r
			\quad\text{and}\quad
			V(P)\cap[\alpha,\beta]=\emptyset.
		\end{equation}
		If $q\le3$, then $|V(P)|\le4$. Since
		\[
		n\ge5r+9\ge5(r+1),
		\]
		\Cref{lem:empty-path-interval}, applied with $M=1$, $N=n$, $L=r+1$, and $a=4$, gives the required interval.
		
		Suppose that $q\ge4$. By \Cref{lem:r<j<n-r}, we must have $u_q>n-r$. Hence, \Cref{cl:u<r} shows that at most three vertices of $P$ lie in $[r,n]$. Since
		\[
		n-r+1\ge4(r+1),
		\]
		applying \Cref{lem:empty-path-interval} to $X=V(P)\cap[r,n]$ with $M=r$, $N=n$, $L=r+1$, and $a=3$ again gives an interval satisfying \eqref{eq:empty-source-interval}.
		
		Choose such an interval $[\alpha,\beta]$. By \Cref{lem:del-graph-iso}, for every $k\in[\alpha,\beta]$, the map $\sigma_k$ induces an isomorphism
		\[
		G_n\setminus[\alpha,\beta]
		\xrightarrow{\ \cong\ }
		G_{n+1}\setminus[\alpha,\beta+1].
		\]
		Since $P$ is minimal in $G_n$ and avoids $[\alpha,\beta]$, it is minimal in the induced subgraph $G_n\setminus[\alpha,\beta]$. Therefore, $\sigma_k(P)$ is a minimal path in $G_{n+1}\setminus[\alpha,\beta+1]$, and hence in $G_{n+1}$. Moreover, since $\sigma_k$ is strictly increasing, it preserves the order of the vertices. Thus, admissibility is preserved as well. It follows that
		\begin{equation}
			\label{eq:forward-path-length}
			\ell(G_n)\le\ell(G_{n+1})
			\quad\text{and}\quad
			\ell'(G_n)\le\ell'(G_{n+1}).
		\end{equation}
		
		We now prove the reverse inequalities. Let
		$Q=(v_0,\dots,v_q)$
		be a minimal path in $G_{n+1}$ with $v_0<v_q$. We claim that there exist $\alpha,\beta\in[n]$ such that
		\begin{equation}
			\label{eq:empty-target-interval}
			\beta-\alpha\ge r
			\quad\text{and}\quad
			V(Q)\cap[\alpha,\beta+1]=\emptyset.
		\end{equation}
		If $q\le3$, then $|V(Q)|\le4$. Since
		\[
		n+1\ge5r+10=5(r+2),
		\]
		\Cref{lem:empty-path-interval}, applied with $M=1$, $N=n+1$, $L=r+2$, and $a=4$, gives an interval $[\alpha,\beta']$ of $r+2$ consecutive integers disjoint from $Q$. Setting $\beta=\beta'-1$, we obtain \eqref{eq:empty-target-interval}.
		
		Suppose that $q\ge4$. Applied to $G_{n+1}$, \Cref{lem:r<j<n-r} yields $v_q>(n+1)-r$. Hence, \Cref{cl:u<r} shows that at most three vertices of $Q$ lie in $[r,n+1]$. Furthermore,
		\[
		(n+1)-r+1=n-r+2\ge4(r+2).
		\]
		Applying \Cref{lem:empty-path-interval} to $X=V(Q)\cap[r,n+1]$ with $M=r$, $N=n+1$, $L=r+2$, and $a=3$, we again obtain an interval $[\alpha,\beta']$ of $r+2$ consecutive integers disjoint from $Q$. Set $\beta=\beta'-1$. Then $\beta\le n$, $\beta-\alpha=r$, and \eqref{eq:empty-target-interval} holds.
		
		For every $k\in[\alpha,\beta]$, \Cref{lem:del-graph-iso} gives an isomorphism
		\[
		\sigma_k\colon
		G_n\setminus[\alpha,\beta]
		\xrightarrow{\ \cong\ }
		G_{n+1}\setminus[\alpha,\beta+1].
		\]
		Since $Q$ avoids $[\alpha,\beta+1]$, it is minimal in the induced subgraph $G_{n+1}\setminus[\alpha,\beta+1]$. Thus, $\sigma_k^{-1}(Q)$ is a minimal path in $G_n\setminus[\alpha,\beta]$, and hence in $G_n$. Since the inverse of an increasing map is increasing on its image, admissibility is preserved as well. Consequently,
		\begin{equation}
			\label{eq:reverse-path-length}
			\ell(G_n)\ge\ell(G_{n+1})
			\quad\text{and}\quad
			\ell'(G_n)\ge\ell'(G_{n+1}).
		\end{equation}
		Combining \eqref{eq:forward-path-length} and \eqref{eq:reverse-path-length} proves part \textup{(ii)}.
	\end{proof}
	
	\begin{rem}
		\Cref{prop:ad-path-length} has consequences for binomial edge ideals associated with graphs in an $\Inc$-invariant chain. Recall that, for a graph $G$ on $[n]$, its \emph{binomial edge ideal}, introduced independently in \cite{HHHKR} and \cite{O11}, is
		\[
		J_G
		=\langle x_i y_j-x_j y_i\mid(i,j)\in E(G)\rangle
		\subseteq
		S\defas\kk[x_1,\dots,x_n,y_1,\dots,y_n].
		\]
		A Gr\"obner basis of $J_G$ indexed by admissible paths is constructed in \cite[Theorem~2.1]{HHHKR}, while the $\Z^n$-graded generic initial ideal of $J_G$ is described in terms of minimal paths in \cite[Theorem~2.1]{CDG}. In both descriptions, a path of length $q$ gives rise to a monomial of degree $q+1$. Therefore, if $\Gc=(G_n)_{n\ge1}$ is an $\Inc$-invariant chain with $\ind(\Gc)=r$, then, for a suitable term order $\le$ and every $n\ge5r$, the ideal $\ini_{\le}(J_{G_n})$ is generated in degrees at most $4$, whereas $\gin_{\le}(J_{G_n})$ is generated in degrees at most $6$. Moreover, for every $n\ge5r+9$, the maximal degrees of the minimal generators of these two ideals are independent of $n$.
	\end{rem}

	\section{Appendix}
	
	This section is devoted to the proofs of the technical results stated in \Cref{subsec.InvChain}.
	
	We first prove \Cref{lem:induced-chain}. Recall that for each $k\ge0$, the map $\sigma_k\in\Inc$ is defined in \eqref{eq.sigma}. In particular, we have $\sigma_0(i)=i+1$ for all $i\in\N$.
	
	\begin{proof}[Proof of \Cref{lem:induced-chain}]
		(i) Set
		$\tilde{r}=r+\alpha-1.$
		Since $U_{n+\beta}=[\alpha,n]$, we have
		\[
		H_n=G_{n+\beta}[\alpha,n]
		\qquad\text{for every }n\ge \tilde{r}.
		\]
		We first prove that
		\begin{equation}
			\label{eq:H-stability}
			E(H_n)=\Inc_{\tilde{r},n}\bigl(E(H_{\tilde{r}})\bigr)
			\qquad\text{for every }n\ge \tilde{r}.
		\end{equation}
		Let $(k,l)\in E(H_n)$. Then $(k,l)\in E(G_{n+\beta})$. By \Cref{lem:E(Gn)}, there exists
		$(i,j)\in E(G_r)$ such that
		\[
		0\le x\defas k-i\le y\defas l-j\le n+\beta-r.
		\]
		Put
		\[
		t=n-\tilde{r}=n-r-\alpha+1
		\]
		and define
		\[
		q=\max\{0,\alpha-i,y-t\},
		\qquad
		p=\max\{0,\alpha-i,x-y+q\}.
		\]
		We claim that
		\begin{equation}
			\label{eq:p-q-bounds}
			0\le p\le q\le\alpha+\beta-1.
		\end{equation}
		Indeed, since $k=i+x\ge\alpha$, we have $\alpha-i\le x\le y$.
		Moreover,
		\[
		y-t
		\le n+\beta-r-(n-r-\alpha+1)
		=\alpha+\beta-1.
		\]
		It follows that $q\le y$ and $q\le\alpha+\beta-1$. Consequently,
		$p\le q$ because $x-y+q\le q$, and $p\le x$ because
		$q\le y$.
		
		Set
		\[
		i'=i+p,\qquad j'=j+q.
		\]
		By \eqref{eq:p-q-bounds} and \Cref{lem:E(Gn)}, we have
		$(i',j')\in E(G_{\tilde{r}+\beta}),$
		since
		\[
		0\le p\le q\le\alpha+\beta-1 =(\widetilde{r}+\beta)-r.
		\]
		Moreover, $i'\ge\alpha$ and $j'\le \tilde{r}$. The first inequality is clear, while the second follows from
		\[
		j\le \tilde{r},\qquad
		j+\alpha-i\le r+\alpha-1=\tilde{r},
		\qquad
		j+y-t=l-t\le \tilde{r}.
		\]
		Thus, $(i',j')\in E(H_{\tilde{r}})$.
		
		Now the definitions of $p$ and $q$ give
		\[
		0\le x-p\le y-q\le t=n-\tilde{r}.
		\]
		Therefore, \Cref{lem:E(Gn)} shows that
		\[
		(k,l)\in\Inc_{\tilde{r},n}\bigl(E(H_{\tilde{r}})\bigr).
		\]
		This proves one inclusion in \eqref{eq:H-stability}.
		
		Conversely, suppose that
		$(k,l)\in\Inc_{\tilde{r},n}\bigl(E(H_{\tilde{r}})\bigr).$
		Then there exists $(i',j')\in E(H_{\tilde{r}})\bigr)$ such that
		\[
		0\le k-i'\le l-j'\le n-\tilde{r}.
		\]
		Since $(i',j')\in E(G_{\tilde{r}+\beta})$, there exists $(i,j)\in E(G_r)$ with
		\[
		0\le i'-i\le j'-j\le\alpha+\beta-1.
		\]
		Adding the preceding inequalities gives
		\[
		0\le k-i\le l-j
		\le (n-\tilde{r})+(\alpha+\beta-1)
		=n+\beta-r.
		\]
		Hence, $(k,l)\in E(G_{n+\beta})$. Moreover,
		$\alpha\le k<l\le n$, so $(k,l)\in E(H_n)$. This proves
		\eqref{eq:H-stability}.
		
		It follows that $\Hc$ is $\Inc$-invariant and
		$\ind(\Hc)\le \tilde{r}.$
		On the other hand, the reverse inequality $\ind(\Hc)\ge \tilde{r}$ is immediate from the definition of $\Hc$. Hence, $\ind(\Hc)= \tilde{r}=r+\alpha-1$.
		
		(ii) Let $s=\spi(\Gc)$. Since $E(H_{\tilde{r}})\subseteq E(G_{\tilde{r}+\beta})\subseteq E(G_{\infty})$,
		it follows from \Cref{cor:G_infty} that
		$\spi(\Hc)\ge s.$
		Conversely, choose $(i,j)\in E(G_r)$ with $j-i=s$. Then
		\[
		(i+\alpha-1,j+\alpha-1)\in E(H_{\tilde{r}}),
		\]
		which yields $\spi(\Hc)\le s.$
		Therefore,
		$\spi(\Hc)=s.$

		(iii) 
		Consider the map $\tau=\sigma_0^{\alpha-1}$. Since
		\[
		\tau(i)=i+\alpha-1
		\quad\text{for all } i\in\N,
		\]
		we have $\tau\in\Inc_{n,n+\alpha-1}$ for every $n\in\N$. It follows that
		\begin{align*}
			\tau(G_m)&\subseteq \Inc_{m,m+\alpha-1}(G_m)\subseteq G_{m+\alpha-1},\\
			\tau(K_n)&=K_{[\alpha,n+\alpha-1]}
		\end{align*}
		for all $m,n\in \N$. Thus,
		\begin{equation}
			\label{eq:tau}
			\tau(G_m\cap K_n)
			\subseteq G_{m+\alpha-1}\cap K_{[\alpha,n+\alpha-1]}
			\quad\text{for all } m,n\in \N.
		\end{equation}
		In particular, when $m=n=r$, we obtain
		\begin{align*}
			\tau(G_r)&=\tau(G_r\cap K_r)\subseteq G_{r+\alpha-1}\cap K_{[\alpha,r+\alpha-1]}\\
			&\subseteq G_{\tilde{r}+\beta}\cap K_{[\alpha,\tilde{r}]}
			=H_{\tilde{r}}.
		\end{align*}
		Since $\tau$ is injective, it induces an injection
		\[
		\tau\colon E(G_r)\longrightarrow E(H_{\tilde{r}}).
		\]
		As $V(H_{\tilde{r}})=[\alpha,{\tilde{r}}]$, we have $|V(H_{\tilde{r}})|=r$.
		It follows that
		\[
		\me(\Hc)
		=\binom{r}{2}-|E(H_{\tilde{r}})|
		\le\binom{r}{2}-|E(G_r)|
		=\me(\Gc).
		\]
		
		Suppose now that $\beta\ge1$ and $\me(\Hc)=\me(\Gc)$. Set
		\[
		F_r=G_{r+1}\cap K_r.
		\]
		Then $G_r\subseteq F_r$. Because $\beta\ge1$, applying \eqref{eq:tau} with $m=r+1$ and $n=r$ yields
		\begin{align*}
			\tau(F_r)
			&\subseteq G_{r+\alpha}\cap K_{[\alpha,r+\alpha-1]}
			\subseteq G_{\tilde{r}+\beta}\cap K_{[\alpha,\tilde{r}]}=H_{\tilde{r}}.
		\end{align*}
		Consequently,
		\[
		\tau\bigl(E(G_r)\bigr)
		\subseteq\tau\bigl(E(F_r)\bigr)
		\subseteq E(H_{\tilde{r}}).
		\]
		The equality $\me(\Hc)=\me(\Gc)$ gives
		$|E(H_{\tilde{r}})|=|E(G_r)|.$
		Hence, all the preceding inclusions are equalities, and thus
		$F_r=G_r.$
		Consider the chain $\Fc=(F_n)_{n\ge1}$ defined by
		\[
		F_n=
		\begin{cases}
			\emptyset,&n<r,\\
			G_{n+1}\cap K_n,&n\ge r.
		\end{cases}
		\]
		Applying part (i) with $\alpha=\beta=1$, we obtain
		\[
		F_n=\Inc_{r,n}(F_r)
		\quad\text{for every }n\ge r.
		\]
		Since $F_r=G_r$ and $\ind(\Gc)=r$, it follows that
		\[
		G_{n+1}\cap K_n=F_n=G_n
		\qquad\text{for every }n\ge r.
		\]
		By \Cref{lem:saturation-characterization}, $\Gc$ is eventually
		saturated. 
	\end{proof}

	\begin{rem}
		In \Cref{lem:induced-chain}, if $\beta=0$, the equality $\me(\Hc)=\me(\Gc)$ does not necessarily imply that $\Gc$ is eventually saturated. For example, consider a chain $\Gc$ with $\ind(\Gc)=3$ and 
		$$E(G_3)=\{(1,2),(1,3)\}.$$ 
		Since $(2,3)\in E(G_4\cap K_3)\setminus E(G_3)$, this chain is not eventually saturated. Let $U_n=[2,n]$ for $n\ge 3$ and define the chain $\Hc$ as in \Cref{lem:induced-chain}. Then 
		$$E(H_4)=\{(2,3),(2,4)\}=\sigma_0(E(G_3)),$$ 
		and therefore $\me(\Hc)=\me(\Gc)$.
	\end{rem}
	
	%------------------------------------------------
	
	Finally, we prove \Cref{lem:del-graph-iso}. To this end, we require the following special case of \cite[Proposition~3.1]{HHLNN}.
	
	\begin{lem}
		\label{lem_decomposition}
		Let $\Gc=(G_n)_{n\ge1}$ be an $\Inc$-invariant chain of graphs with $\ind(\Gc)=r$. Then for any subset $\Lambda\subseteq\{0,1,\dots,n\}$ with $|\Lambda|=r+1$, it holds that
		\[
		E(G_{n+1})=\bigcup_{k\in \Lambda}\sigma_k(E(G_n))
		\quad
		\text{for all } n\ge r.
		\]
	\end{lem}

	\begin{proof}[Proof of \Cref{lem:del-graph-iso}]
		Fix $n\ge r.$ For any $k\in [\alpha,\beta]$, we have
		\[
		\sigma_k(V(G_{n}\setminus [\alpha,\beta]))=\sigma_k([n]\setminus [\alpha,\beta])
		=[n+1]\setminus [\alpha,\beta+1] =V(G_{n+1}\setminus [\alpha,\beta+1]).
		\]
		Since $\sigma_k(E(G_{n}))\subseteq E(G_{n+1})$, this implies
		\[
		\sigma_k(E(G_{n}\setminus [\alpha,\beta] ))\subseteq E(G_{n+1}\setminus [\alpha,\beta+1]).
		\]	
		Because $\beta-\alpha\ge r$, it follows from \Cref{lem_decomposition} that
		\[
		E(G_{n+1})=\bigcup_{k\in [\alpha,\beta]}\sigma_k(E(G_n)).
		\]
		Consequently,
		\begin{equation}
			\label{eq:deletion-union}
			E(G_{n+1}\setminus[\alpha,\beta+1])
			=\bigcup_{k\in [\alpha,\beta]}
			\sigma_k\bigl(E(G_n\setminus[\alpha,\beta])\bigr).
		\end{equation}
		To conclude, we show that the sets $\sigma_k(E(G_n\setminus [\alpha,\beta]))$ are independent of $k\in [\alpha,\beta]$. Indeed, let $(i,j)\in E(G_n\setminus [\alpha,\beta])$. Then one of the following cases occurs:
		
		\smallskip
		\emph{Case 1:} $i<j<\alpha$.
		In this case, $\sigma_k(i,j)=(i,j)$ for every $k\in [\alpha,\beta]$.
		
		\smallskip
		\emph{Case 2:} $i<\alpha<\beta <j$.
		Here, $\sigma_k(i,j)=(i,j+1)$ for every $k\in [\alpha,\beta]$.
		
		\smallskip
		\emph{Case 3:} $\beta <i<j$.
		Here, $\sigma_k(i,j)=(i+1,j+1)$ for every $k\in [\alpha,\beta]$.
		
		\smallskip
		In each case, $\sigma_k(i,j)$ does not depend on $k$. Consequently, \eqref{eq:deletion-union} yields
		\[
		E(G_{n+1}\setminus[\alpha,\beta+1])
		=\sigma_k\bigl(E(G_n\setminus[\alpha,\beta])\bigr).
		\] 
		Hence, $\sigma_{k}$ induces the desired isomorphism for every $k\in [\alpha,\beta]$.
	\end{proof}
	
	\section*{Acknowledgments}
	
	We are grateful to Hop D. Nguyen for allowing us to include his formulation of \Cref{thm:f-vector-IN}. The first and third authors were supported by the Vietnam National Foundation for Science and Technology Development (NAFOSTED) under grant number 101.04-2025.49. Part of this work was carried out while the second and third authors were visiting the Vietnam Institute for Advanced Study in Mathematics (VIASM). We thank VIASM for its hospitality and financial support.

\end{document}